\documentclass[12pt]{iopart}
\usepackage[dvips]{graphicx}
\usepackage{amsfonts}
\usepackage{mathrsfs}
\usepackage[dvipsnames,usenames]{color}
\usepackage{amssymb}
\usepackage{graphicx}
\usepackage{times}
\usepackage{color}

\newtheorem{theorem}{Theorem}
\newenvironment{proof}{\vspace{1ex}\noindent{\em Proof.}\hspace{0.5em}}
{\hfill\qed\vspace{1ex}}
\def\qed{\hfill \vrule height 6pt width 6pt depth 0pt}

\begin{document}

\title[Precise Mechanical Control]{Achieving Precise Mechanical Control in
Intrinsically Noisy Systems}
\author{{\small Wenlian Lu$^{1,2,3}$,\quad Jianfeng Feng$^{1,2}$,\quad
Shun-ichi Amari$^{3}$\quad and David Waxman$^{1}$}}

\begin{abstract}
How can precise control be realised in intrinsically noisy systems? Here, we
develop a general theoretical framework that provides a way to achieve
precise control in signal-dependent noisy environments. When the control
signal has Poisson or supra-Poisson noise, precise control is not possible.
If, however, the control signal has sub-Poisson noise, then precise control
is possible. For this case, the precise control solution is not a function,
but a rapidly varying random process that must be averaged with respect to a
governing probability density functional. Our theoretical approach is
applied to the control of straight-trajectory arm movement. Sub-Poisson
noise in the control signal is shown to be capable of leading to precise
control. Intriguingly, the control signal for this system has a natural
counterpart, namely the bursting pulses of neurons --trains of
Dirac-delta functions-- in biological systems to achieve precise control
performance.
\end{abstract}

\maketitle

\address{\small $^{1}$ Centre for Computational Systems Biology and School of Mathematical
Sciences, Fudan University, Shanghai 200433, China}
\address{\small $^{2}$ Centre for Scientific Computing, the University of Warwick, Coventry CV4
7AL, UK}
\address{\small $^{3}$Brain Science Institute, RIKEN, Wako-shi,
Saitama 351-0198, Japan} \ead{jianfeng64@gmail.com}


\section{Introduction}

Many mechanical and biological systems are controlled by signals which
contain noise. This poses a problem. The noise apparently corrupts the
control signal, thereby preventing precise control. However, precise control
can be realised, despite the occurrence of noise, as has been demonstrated
experimentally in biological systems. For example, in neural-motor control,
as reported in \cite{OLB}, the movement error is believed to be mainly due
to inaccuracies of the neural-sensor system, and not associated with the
neural-motor system.

The minimum-variance principle proposed in \cite{Harris,Harris1} has greatly
influenced the theoretical study of biological computation. Assuming
the magnitude of the noise in a system depends strongly on the magnitude of
the signal, the conclusion of \cite{Harris,Harris1} is that a biological
system is controlled by minimising the execution error.

A key feature of the control signal in a biological system is that
biological computation often only takes on a finite number of values. For
example, `bursting' neuronal pulses in the neural-motor system control seem
very likely to have only three states, namely inactive, excited, and
inhibited. This kind of signal (neuronal pulses) can be abstracted as a
dynamic trajectory which is zero for most of the time, but intermittently
takes a very large value. Generally, this kind of signal looks like a train
of irregularly spaced Dirac-delta functions. In this work we shall
theoretically investigate the way signals in realistic biological
systems are associated with precise control performance. We shall use
bursting neuronal pulse trains as a prototypical example of this phenomenon.

In a biological system, noise is believed to be inevitable and essential; it
is a part of a biological signal and, for example, the magnitude of the
noise typically depends strongly  on the magnitude of the
signal \cite{Harris,Harris1}. One characteristic of the noise in a system is
the dispersion index, $\alpha $, which describes the statistical regularity
of the control signal. When the variance in the control signal is
proportional to the $2\alpha $-th power of the mean control signal, the
dispersion index of the control noise is said to be $\alpha $. It was shown
in \cite{Harris,Harris1} and elsewhere (e.g., \cite{Feng1,Feng11}) that an
optimal solution of analytic form can be found when the stochastic control
signal is supra-Poisson, i.e., when $\alpha \geq 0.5$. However, the
resulting control is not precise and a non-zero execution error arises. In
recent papers, a novel approach was proposed to find the optimal solution
for control of a neural membrane \cite{Feng2}, and a model of saccadic
eye movement \cite{Ross}. It was shown that if the noise of the control
signal is more regular than Poisson process (i.e., if it is sub-Poisson,
with $\alpha <0.5$), then the execution error can be shown to reduce towards
zero \cite{Feng2,Ross}. This work employed the theory of Young measures \cite%
{Young,Young1}, and involved a very specific sort of solution (a `relaxed
optimal parameterized measure solution'). We note that \textit{many}
biological signals are more regular than a Poisson process: e.g., within
in-vivo experiments, it has often been observed that neuronal pulse signals
are sub-Poisson in character ($\alpha <0.5$) \cite{Sub,Chris}. However, in
\cite{Feng2,Ross}, only a one-dimensional linear model was studied in
detail. Thus the results and methods cannot be applied to the control of
general dynamical systems. The work of \cite{Feng2,Ross} however, leads
to a much harder problem: the general mathematical link between the
regularity of the signal's noise and the control performance that can be
achieved.

In the present work we establish some general mathematical principles
linking the regularity of the noise in a control signal with the
precision of the resulting control performance, for \textit{general}
nonlinear dynamical systems of high dimension. We establish a general
theoretical framework that yields \textit{precise control} from a \textit{%
noisy controller} using modern mathematical tools. The control signal is
formulated as a Gaussian (random) process with a signal-dependent variance.
Our results show that if the control signal is more regular than a
Poisson process (i.e., if $\alpha <0.5$), then the control optimisation
problem naturally involves solutions with a specific singular character
(parameterized measure optimal solutions), which can achieve precise control
performance. In other words, we show how to achieve results where the
variance in control performance can be made arbitrarily small. This is in
clear contrast to the situation where the control signals are Poisson or
more random than Poisson ($\alpha \geq 0.5$), where the optimal control
signal is an ordinary function, not a parameterized measure, and the
variance in control performance does not approach zero. The new results
can be applied to a large class of control problems in nonlinear dynamical
systems of high dimension. We shall illustrate the new sort of solutions
with an example of neural-motor control, given by the control of
straight-trajectory arm movements, where neural pulses act as the control
signals. We show how pulse trains may be realised in nature which lead
towards the optimisation of control performance.

\section{Model and Mathematical Formulation}

To establish a theoretical approach to the problem of noisy control, we
shall consider the following general system
\begin{equation}
\frac{dx}{dt}=a(x(t),t)+b(x(t),t)u(t)  \label{general}
\end{equation}%
where: $t$ is time ($t\geq 0$), $x(t)=[x_{1}(t),\ldots ,x_{n}(t)]^{\top }$
is a column vector of `coordinates' describing the state of the system to be
controlled (a $\top $-superscript denotes transpose) and $%
u(t)=[u_{1}(t),\ldots ,u_{m}]^{\top }$, is a column vector of the signals
used to control the $x$ system. The dynamical behaviour of the $x$ system,
in the absence of a control signal, is determined by $a(x,t)$ and $b(x,t)$,
where $a(x,t)$ consists of $n$ functions: $a(x,t)=[a_{1}(x,t),\ldots
,a_{n}(x,t)]^{\top }$ and $b(x,t)$ is an $n\times m$ `gain matrix' with
elements $b_{ij}(x,t)$. The system (\ref{general}) is a generalisation of
the dynamical systems studied in the literature \cite%
{Harris,Harris1,Feng2,Ross}.

As stated above, the control signal, $u(t)$, contains noise. We follow
Harris's work \cite{Harris,Harris1} on signal-dependent noise theory by
modelling the components of the control signal as%
\begin{equation}
u_{i}(t)=\lambda _{i}(t)+\zeta _{i}(t)  \label{ui}
\end{equation}%
where $\lambda _{i}(t)$ is the mean control signal at time $t$ of the $i$'th
component of $u(t)$ and all noise (randomness) is contained in $\zeta _{i}(t)
$. In particular, we take the $\zeta _{i}(t)$ to be an independent
Gaussian white noises obeying $\mathbf{E}\left[ \zeta _{i}(t)\right] =0$%
, $\mathbf{E}\left[ \zeta _{i}(t)\zeta _{j}(t^{\prime })\right] =\sigma
_{i}(t)\sigma _{j}(t^{\prime })\delta (t-t^{\prime })\delta _{ij}$ where $%
\mathbf{E}\left[ \cdot \right] $ denotes expectation, $\delta (\cdot )$
is Dirac-delta function, and $\delta _{ij}$ is Kronecker delta. The
quantities $\sigma _{i}(t)$, which play the role of standard deviations of
the $\zeta _{i}(t)$, are taken to explicitly depend on the mean magnitudes
of the control signals:
\begin{equation}
\sigma _{i}(t)=\kappa _{i}|\lambda _{i}(t)|^{\alpha }  \label{sigma i}
\end{equation}%
where $\kappa _{i}$ is a positive constant and $\alpha $ is the dispersion
index of the control process (described above).

Thus, we can formulate the dynamical system, Eq. (\ref{general}), as a
system of It\^{o} diffusion equations:
\begin{equation}
dx=A(x,t,\lambda )dt+B(x,t,\lambda )dW_{t}  \label{Ito}
\end{equation}%
where: (i) $W_{t}=[W_{1,t},\ldots ,W_{m,t}]^{\top }$ contains $m$
independent standard Wiener processes; (ii) the quantity $A(x,t,\lambda )$
denotes the column vector $[A_{1}(x,t,\lambda ),\ldots ,A_{n}(x,t,\lambda
)]^{\top }$, the $i$'th component of which has the form $A_{i}(x,t,\lambda
)=a_{i}(x,t)+\sum_{j=1}^{m}b_{ij}(x,t)\lambda _{j}$; (iii) the quantity $%
B(x,t,\lambda )$ is the matrix, the $i,j$'th element of which is given by $%
B_{ij}(x,t,\lambda )=b_{ij}(x,t)\kappa _{j}|\lambda _{j}|^{\alpha }$ where $%
i=1,\cdots ,n$ and $j=1,\ldots ,m$. We make the assumption that the range of
each $\xi _{i}$ is bounded: $-M_{Y}\leq \lambda _{i}\leq M_{Y}$ with $M_{Y}$
a positive constant. Let $\Omega =[-M_{Y},M_{Y}]^{m}$ be the region where
the control signal takes values, with $^{m}$ denoting the $m$-order
Cartesian product. Let $\Xi $ be state space of $x$. In this paper we
assume it be bounded.

Let us now introduce the function $\phi (x,t)=[\phi _{1}(x,t),\cdots ,\phi
_{k}(x,t)]^{\top }$, which represents the objective that is to be controlled
and optimised. For example, for a linear output we can take $\phi (x,t)=Cx$
for some $k\times n$ matrix $C$; in the case that we control the magnitude
of $x$, we can take {$\phi (x,t)=\Vert x\Vert _{2}$; we may even allow
dependence on time, for example if the output decays exponentially with
time, we can take $\exp (-\gamma t)x(t)$ for some constant $\gamma>0$. }

The aim of the control problem we consider here is: (i) to ensure the
expected trajectory of the objective $\phi (x(t),t)$ reaches a specified
target at a given time, $T$, and (ii) to minimise the \textit{execution error}
accumulated by the system, during the time, $R$, that the system is required
to spend at the `target' \cite{Harris,Harris1,Feng2,Ross,Win,Sim,Tan,Ike}.
In the present context, we take the motion to start at time $t=0$ and
subject to the initial condition $x(0)$. The target has coordinates $z(t)$
and we need to choose the controller, $u(t)$, so that for the time interval $%
T\leq t\leq T+R $ the expected state of the objective $\phi(x(t),t)$ of the
system satisfies $\mathbf{E}[\phi (x(t),t)]=z(t)$. The \textit{accumulated
execution error} is $\int_{T}^{T+R}\sum_{i}var\left( \phi
_{i}(x(t),t)\right) dt$ and we require this to be minimised.

Statistical properties of $x(t)$ can be written in terms of $p(x,t)$, the
probability density of the state of the system (\ref{Ito}) at time $t$,
which satisfies the Fokker-Plank equation%
\begin{equation}
\frac{\partial p(x,t)}{\partial t}=\mathcal{L}\circ p~;~p(x,0)=\delta
(x-x(0)),~t\in\lbrack0,T+R] \label{FPE}
\end{equation}
with
\begin{eqnarray*}
\mathcal{L}\circ p & =-\sum_{i=1}^{n}\frac{\partial\lbrack
a_{i}(x,t)+\sum_{j=1}^{m}b_{ij}(x,t)\lambda_{j}(t))p(x,t)]}{\partial x_{i}}
\\
& +\frac{1}{2}\sum_{i,j=1}^{n}\frac{\partial^{2}\{\sum_{k=1}^{m}\kappa
_{k}^{2}b_{ik}(x,t)b_{jk}(x,t)|\lambda_{k}(t)|^{2\alpha}p(x,t)\}}{\partial
x_{i}\partial x_{j}}.
\end{eqnarray*}

Three important quantities are the following:

\begin{itemize}
\item[(A)] The \textit{accumulated} \textit{execution error}: $%
\int_{T}^{T+R}\int_{\Xi}\left\Vert\phi(x,t)-z(t)\right\Vert ^{2}p(x,t)dxdt$;

\item[(B)] The \textit{expectation condition} on $x(t)$: $\int_{\Xi}
\phi(x,t)p(x,t)dx=z(t)$, for all $t$ in the interval $T\leq t\leq R+T$;

\item[(C)] The \textit{dynamical equation} of $p(x,t)$ described as (%
\ref{FPE}).
\end{itemize}

\section{The Young Measure Optimal Solution}

To illustrate the idea of the solutions we introduce here, namely Young
measure optimal solutions, we provide a simple example. Consider the
situation where $x$ and $u$ are one-dimensional functions, while $a(x,t)=px$%
, $b(x,t)=q$, $\kappa =1$, $z(t)=z_{0}$ and $\phi (x,t)=x$. Thus (\ref%
{general}) becomes%
\begin{equation}
\frac{dx}{dt}=px+qu.  \label{linear}
\end{equation}%
This has the solution $x(t)=x_{0}\exp (pt)+\int_{0}^{t}\exp (p(t-s))q\lambda
(s)ds+\int_{0}^{t}\exp (p(t-s))q|\lambda (s)|^{\alpha }dW_{s}$. Thus, its
expectation is $\mathbf{E}(x(t))=x_{0}\exp (pt)+\int_{0}^{t}\exp
(p(t-s))q\lambda (s)ds$ and its variance is ${\mathrm{v}ar}%
(x(t))=\int_{0}^{t}\exp (2p(t-s))q^{2}|\lambda (s)|^{2\alpha }ds$. The
solution of the optimisation problem is the minimum of the following
functional:
\begin{eqnarray}
&&H[\lambda ]=\int_{T}^{T+R}\int_{0}^{t}\exp (2p(t-s))q^{2}|\lambda
(s)|^{2\alpha }dsdt  \nonumber \\
&&+\int_{T}^{T+R}\left\{ \gamma (t)[x_{0}\exp (pt)+\int_{0}^{t}\exp
(p(t-s))q\lambda (s)ds-z_{0}]\right\} dt  \nonumber \\
&=&\int_{T}^{T+R}\{[g(t)|\lambda (t)|^{2\alpha }-f(t)\lambda (t)]+\mu (t)\}dt
\label{eg1}
\end{eqnarray}%
with
\begin{eqnarray*}
g(t) &=&\left\{
\begin{array}{ll}
\frac{q^{2}}{2p}\bigg[\exp (2p(T+R-t))-\exp (2p(T-t))\bigg] & t\leq T \\
\frac{q^{2}}{2p}\bigg[\exp (2p(T+R-t))-1\bigg] & T+R\geq t>T%
\end{array}%
\right.  \\
f(t) &=&\left\{
\begin{array}{ll}
-\int_{T}^{T+R}q\gamma (s)\exp (p(s-t))ds & t\leq T \\
-\int_{t}^{T+R}q\gamma (s)\exp (p(s-t))ds & T+R\geq t>T%
\end{array}%
\right.  \\
\mu (t) &=&\gamma (t)[x_{0}\exp (pt)-z_{0}],
\end{eqnarray*}%
for some integrable function $\gamma (t)$, which serves as a Lagrange
auxiliary multiplier.

In the general case, we minimise (A) using (B) and (C) as constraints via
the introduction of appropriate $x$ and $t$ dependent Lagrange multipliers.
This leads to a functional of the mean control signal, $H\left[ \lambda %
\right] $, with the form $H[\lambda ]=\int_{T}^{T+R}h(t,\lambda (t))dt$ (see
below and Appendix A). Let us use $\xi =[\xi _{1},\cdots ,\xi _{m}]^{\top }$
to denote the value of $\lambda (t)$ at a given time of $t$, i.e., ${\xi }%
=\lambda (t)$; $\xi$ will serve as a variable of the Young measure (see
below). We find
\begin{equation}
h(t,\xi )=\sum_{i=1}^{m}[g_{i}(t)|\xi _{i}|^{2\alpha }-f_{i}(t)\xi _{i}]+z(t)
\label{h}
\end{equation}%
where $g_{i}(t)$, $f_{i}(t)$ and $z(t)$ are functions with respect to $t$
but are independent of the variable $\xi $.

The abstract Hamiltonian minimum (maximum) principle (AHMP) \cite{Roub}
provides a necessary condition for the optimal solution of minimising (A)
with (B) and (C), which is composed of the points in the domain of
definition of $\lambda$, namely, $\Omega $, that minimize the function $%
h(t,\xi)$ in (\ref{h}), at each time, $t$, which is named \emph{Hamiltonian
integrand}. This principle tells us that the optimal solution should pick
values of the minimum of $h(t,\xi)$ with respect to $\xi $, for each $t$.

If the control signal is supra-Poisson or Poisson, namely the dispersion
index $\alpha \geq 0.5$, for each $t\in \lbrack 0,T+R]$, the Hamiltonian
integrand $h(t,\xi)$ is convex (or semi-convex) with respect to $\xi$ and so
has a unique minimum point with respect to each $\xi_{i}$. So, the optimal
solution is a deterministic \emph{function} of time: for each $t_{0}$, $%
\lambda _{i}(t_{0})$ can be regarded as picking value at the minimum point
of $h(t,\xi)$ for $t=t_{0}$.

When $\alpha <1/2$, namely when the control signal is sub-Poisson, it
follows that $h(t,\xi )$ is no longer a convex function. Figs. \ref{al} show
the possible minimum points of the term $g_{i}(t)|\xi _{i}|^{2\alpha
}-f_{i}(t)\xi _{i}$ with $g_{i}(t)>0$ and $f_{i}(t)>0$. From the assumption
that the range of each $\xi _{i}$ is bounded, namely $-M_{Y}\leq \xi
_{i}\leq M_{Y}$, it then directly follows, from the form of $h(t,\xi )$,
that the value of $\xi _{i}$ which optimises $h(t,\xi )$ is not unique;
there are \textit{three} possible minimum values: $-M_{Y}$, $0$, and $M_{Y}$%
, as shown in Table \ref{maximum points}. So, no explicit function $\lambda
(t)$ exists which is the optimal solution of the optimisation problem
(A)-(C). However, an infinimum of (\ref{h}) does exist.

Proceeding intuitively, we first make an arbitrary choice of one of the
three optimal values for $\xi_{i}$ (namely one of $-M_{Y}$, $0$, and $M_{Y}$%
) and then \textit{average} over all possible choices at each time. With $%
\eta_{t,i}(\xi)$ the probability density of $\xi_{i}$ at time $t$, the
average is carried out using the distribution (probability density
functional) $\eta[\lambda]\propto{\prod\nolimits_{t,i}}\eta_{t,i}(\xi)$
which represents independent choices of the control signal at each time.
Thus, for example, the functional $H[\lambda]$ becomes \textit{functionally
averaged} over $\lambda(\cdot)$ according to {$\int_{T}^{T+R}\left( \int
h(t,\lambda )\eta[\lambda]d[\lambda]\right) dt$. The optimisation problem
has thus shifted from determining a function (as required when $\alpha\geq
1/2$) to determining a \textit{probability density functional}, $\eta[\lambda%
]$. This intuitively motivated procedure is confirmed by optimisation
theory- and this leads us to Young measure theory. }

Let us spell it out in a mathematical way. Young measure theory \cite%
{Young,Young1} provides a solution to an optimization problem where a
solution, which was a function, becomes a linear functional of a
parameterized measure. By way of explanation, a function, $\lambda (t)$,
yields a single value for each $t$, but a parameterized measure $\{\eta
_{t}(\cdot )\}$ yields a \textit{set of values} on which a measure (i.e., a
weighting) $\eta _{t}(\cdot )$ is defined for each $t$. A \textit{functional}
with respect to a parameterized measure can be treated in a similar way to a
solution that is an explicit function, by averaging over the set of values
of the parameterized measure at each $t$. In detail, a functional of the
form $H[\lambda ]=\int_{0}^{T}h(t,\lambda (t))dt$, of an explicit function, $%
\lambda (t)$, can have its definition extended to a parameterized measure $%
\eta _{t}(\cdot )$, namely $H[\eta ]=\int_{0}^{T}\int_{\Omega }h(t,\xi )\eta
_{t}(d\xi )dt$. In this sense, an explicit function can be regarded as a
special solution that is a `parameterized concentrated measure' (i.e.,
involving a Dirac-delta function) in that we can write $H[\lambda
]=\int_{0}^{T}\int_{\Omega }h(t,\xi )\delta (\xi -\lambda (t))d\xi dt$.
Thus, we can make the equivalence between the explicit function $\lambda (t)$
and a parameterized concentrated measure $\{\delta (\xi -\lambda (t))\}_{t}$
and then replace this concentrated measure, when appropriate, by a Young
measure.

Technically, a Young measure is a class of parameterized measures that are
relatively weak*-compact such that the Lebesgue function space can be
regarded as its dense subset in the way mentioned above. Thus, by enlarging
the solution space from the function space to the (larger) Young measure
space, we can find a solution in the larger space and the minimum value of
the optimisation problem, in the Young measure space, coincides with the
infinimum in the Lebesgue function space.

For any function $r(x,t,\xi)$, we denote a symbol $\cdot$ as the inner
product of $r(x,t,\xi)$ over the parameterized measure $\eta_{t}(d\xi)$, by
averaging $r(x,t,\xi)$ with respect to $\xi$ via $\eta_{t}(\cdot)$. That is
we define $r(x,t,\xi)\cdot\eta_{t}$ to represent $\int_{\Omega}r(x,t,\xi)%
\eta_{t}(d\xi )$. In this way we can rewrite the optimisation problem
(A)-(C) as:
\begin{equation}
\left\{
\begin{array}{ll}
\min_{\eta } & \int_{T}^{T+R}\int_{\Xi }\Vert \phi (x,t)-z(t)\Vert
^{2}p(x,t)dxdt \\
\mathrm{subject~to} & \frac{\partial p(x,t)}{\partial t}=[(\mathcal{L}\cdot
\eta )\circ p](x,t),~on~[0,T]\times \Xi ,~p(x,0)=p_{0}(x), \\
& x\in \Xi,~t\in[0,T+R] \\
& \int_{\Xi }\phi (x,t)p(x,t)dx=z(t),~on~[T,T+R],~\eta \in \mathcal{Y}.%
\end{array}%
\right.  \label{ROP}
\end{equation}%
Here, $\mathcal{Y}$ denotes the Young measure space, which is defined on the
state space $\Xi$ with $t\in[0,T+R]$, while $\eta =\{\eta _{t}(\cdot )\}$
denotes a shorthand for the Young measure associated with control; $(%
\mathcal{L}\cdot \eta )\circ p$ is defined as
\begin{eqnarray*}
&&[\mathcal{L}\cdot \eta ]\circ p=\mathcal{\int }_{\Omega}\mathcal{L}%
(t,x,\xi )\circ p(x,t)\eta _{t}(d\xi ) \\
&=&\int_{\Omega }\bigg\{-\sum_{i=1}^{n}\frac{\partial \lbrack A_{i}(x,t,\xi
)p(x,t)]}{\partial x_{i}} \\
&&+\frac{1}{2}\sum_{i,j=1}^{n}\frac{\partial ^{2}\{[B(x,t,\xi )B^{\top
}(x,t,\xi )]_{ij}p(x,t)\}}{\partial x_{i}\partial x_{j}}\bigg\}\eta
_{t}(d\xi ) \\
&=&\int_{\Omega}\bigg\{-\sum_{i=1}^{n}\frac{\partial\lbrack
a_{i}(x,t)+\sum_{j=1}^{m}b_{ij}(x,t)\xi_{j})p(x,t)]}{\partial x_{i}} \\
&& +\frac{1}{2}\sum_{i,j=1}^{n}\frac{\partial^{2}\{\sum_{k=1}^{m}\kappa
_{k}^{2}b_{ik}(x,t)b_{jk}(x,t)|\xi_{k}|^{2\alpha}p(x,t)\}}{\partial
x_{i}\partial x_{j}}\bigg\}\eta_{t}(d\xi).
\end{eqnarray*}%
So, we can study the relaxation problem (\ref{ROP}) instead of the original
one, (A)-(C). We assume that the constraints in (\ref{ROP}) admit a nonempty
set of $\lambda(t)$, which guarantees that the problem (\ref{ROP}) has a
solution. We also assume the existence and uniqueness of the Cauchy problem
of the Fokker-Plank equation (\ref{FPE}).

The abstract Hamiltonian minimum (maximum) principle (Theorem 4.1.17 \cite%
{Roub}) also provides a similar necessary condition for the Young measure
solution of (\ref{ROP}), if it admits a solution, that is composed of the
points in $\Omega $ which minimise the integrand of the underlying `abstract
Hamiltonian'. By employing variational calculus with respect to the Young
measure, we can derive the form (\ref{h}), for the Hamiltonian integrand.
See Appendix A for details.

Via this principle, the problem conceptively reduces to finding the minimum
points of $h(t,\xi )$. From Table \ref{maximum points}, for a sufficiently
large $M_{Y} $, it can be seen that, if $\alpha <0.5$, then the minimum
points for each $t$ with $g_{i}(t)>0$ may be TWO points $\{0,M_{Y}\}$ or $%
\{-M_{Y},0\}$. Hence, in the case of $\alpha <0.5$, the optimal solution of (%
\ref{ROP}) is a measure on $\{M_{Y},0\}$ or $\{-M_{Y},0\}$. This implies
that the optimal solution of (\ref{ROP}) should have the following form $%
\eta _{t}(\cdot )=\eta _{1,t}(\cdot )\times ,\cdots ,\eta _{m,t}(\cdot )$,
where $\times $ stands for the Cartesian product, and each $\eta _{i,t}$ we
adopt is a measure on $\{-M_{Y},0,M_{Y}\}$:
\begin{equation}
\eta _{i,t}(\cdot )=\mu _{i}(t)\delta _{M_{Y}}(\cdot )+\nu _{i}(t)\delta
_{-M_{Y}}(\cdot )+[1-\mu _{i}(t)-\nu _{i}(t)]\delta _{0}(\cdot )
\label{functional density}
\end{equation}%
where $\mu _{i}(t)$ and $\nu _{i}(t)$ are non-negative weight functions. The
optimisation problem corresponds to the determination of the $\mu_{i}(t)$
and $\nu _{i}(t)$. Averaging with respect to $\eta $ corresponds to the
optimal control signal when the noise is sub-Poisson ($\alpha <0.5$). This
assignment of a probability density for the solution at each time is known
in the mathematical literature as a Young Measure \cite{Roub,Young,Young1}.
For all $i$ and $t$, the weight functions satisfy: (i) $\mu _{i}(t)+\nu
_{i}(t)\leq 1$ and (ii) $\mu_{i}(t)\nu_{i}(t)=0$ (owing to the properties
mentioned above that $h(t,\xi _{i})$ cannot simultaneously have both $M_{Y}$
and $-M_{Y}$ as optimal).

Consider the simple one-dimensional system (\ref{linear}). We shall provide
the explicit form of the optimal control signal $u(t)$ as a Young measure.
Taking expectation for both sides in (\ref{linear}), we have
\[
\frac{d\mathbf{E}~(x)}{dt}=p\mathbf{E}(x)+q\lambda (t).
\]%
Since we only minimise the variance in $[T,T+R]$ for some $T>0$ and $R>0$,
the control signal $u(t)$ for $t\in \lbrack 0,T)$ is picked so that the
expectation of $x(t)$ can reach $z_{0}$ at the time $t=T$. After some simple
calculations, we find a deterministic $\lambda (t)$ as follows:
\[
\lambda (t)=\frac{z_{0}-x_{0}\exp (pT)}{Tq}\exp (p(-T+t)),~t\in \lbrack 0,T]
\]%
such that $\mathbf{E}(x(T))=z_{0}$. Then we pick $\lambda (t)=-pz_{0}/q$ for
$t\in \lbrack T,T+R]$ such that $d\mathbf{E}(x(t))/dt=0$ for all $t\in
\lbrack T,T+R]$. Hence, $\mathbf{E}(x(t))=z_{0}$ for all $t\in \lbrack
T,T+R] $. In the interval $[T,T+R]$, as discussed above, for a sufficiently
large $M_{Y}$, the optimal solution of $\lambda (t)$ should be a Young
measure that picks values in $\{0,M_{Y},-M_{Y}\}$. To sum up, we can
construct the optimal $\lambda (t)$ as follows:
\[
\eta _{t}(\cdot )=\left\{
\begin{array}{ll}
\delta _{\lambda (t)}(\cdot ) & t\in \lbrack 0,T) \\
\delta _{M_{Y}}(\cdot )\frac{-pz_{0}}{q~M_{Y}}+\delta _{0}(\cdot )[1+\frac{%
pz_{0}}{q~M_{Y}}] & t\in \lbrack T,T+R]~\mathrm{if}~-pz_{0}/q>0~\mathrm{or}
\\
\delta _{-M_{Y}}(\cdot )\frac{pz_{0}}{q~M_{Y}}+\delta _{0}(\cdot )[1-\frac{%
pz_{0}}{q~M_{Y}}] & t\in \lbrack T,T+R]~\mathrm{if}~pz_{0}/q\geq 0.%
\end{array}%
\right.
\]%
It can be seen that in $[0,T)$, $\eta _{t}(\cdot )$ is in fact a
deterministic function as the same as $\lambda (t)$.

\section{Precise Control Performance}

We now illustrate the control performance when the noise is sub-Poisson. For
the general nonlinear system (\ref{general}), we cannot obtain an explicit
expression for the probability density functional $\eta[\lambda]$, Eq. (\ref%
{functional density}), or the value of the variance (execution error).
However, we can adopt a \textit{non optimal} probability density functional
which illustrates the property of the exact system, that the execution error
becomes arbitrarily small when the bound of the control signal, $M_{Y}$,
becomes arbitrarily large. In the simple case (\ref{linear}), we note that
if there is a $\hat {u}(t)$, such that $\mathbf{E}(x(t))=z(t) $, then the
variance becomes, expressed by Young measure $\hat{\eta}(\cdot)$,
\begin{eqnarray*}
{\mathrm{v}ar}(x(t)) &
=\int_{0}^{t}\int_{-M_{Y}}^{M_{Y}}\exp(2p(t-s))q^{2}|\xi|^{2\alpha}\hat{\eta}%
(d\xi)ds \\
& =\int_{0}^{t}\exp(2p(t-s))q^{2}M_{Y}^{2\alpha}\frac{|\hat{u}(s)|}{M_{Y}}ds,
\end{eqnarray*}
which converges to zero as $M_{Y}\rightarrow\infty$, due to $\alpha<0.5$.
That is, the minimised execution error can be arbitrarily small if the bound
of the control signal, $M_{Y}$, goes sufficiently large.

In fact, this phenomenon holds for general cases. The non optimal
probability density functional is motivated by assuming that there is a
deterministic control signal $\hat{u}(t)$ which controls the dynamical
system
\begin{equation}
\frac{d\hat{x}}{dt}=A(\hat{x},t,\hat{u}(t))  \label{ds}
\end{equation}
which is the original system (\ref{general}), \textit{with the noise removed}%
. The deterministic control signal $\hat{u}(t)$ causes $\hat{x}(t)$ to
precisely achieve the target trajectory $\hat{x}(t)=z(t)$ for $T\leq t\leq
T+R$.

Then, we add the noise with the signal-dependent variance: $\sigma
_{i}=\kappa _{i}|\lambda _{i}|^{\alpha }$ with some $\alpha <0.5$, which
leads a stochastic differential equation, $dx=A(x,t,\lambda
(t))dt+B(x,t,\lambda (t))dW_{t}$. The non optimal probability density that
is appropriate for time $t$, namely $\hat{\eta}_{t,i}(\xi )$, is constructed
to have a mean over the control values $\{-M_{Y},0,M_{Y}\}$, which equals $%
\hat{u}(t)$. This probability density is
\begin{equation}
\hat{\eta}_{i,t}(\lambda _{i})=\frac{|\hat{u}_{i}(t)|}{M_{Y}}\delta _{\sigma
(t)M_{Y}}(\lambda _{i})+(1-\frac{|\hat{u}_{i}(t)|}{M_{Y}})\delta
_{0}(\lambda _{i})  \label{functional density1}
\end{equation}%
where $\sigma (t)=\mathrm{sign}(\hat{u}_{i}(t))$ and, by definition, $%
\hat{u}_{i}(t)=\int_{-M_{Y}}^{M_{Y}}\lambda _{i}\hat{\eta}_{i,t}(\lambda
_{i})d\lambda _{i}$. We establish in Appendix B that the expectation
condition ((B) above) holds asymptotically when $M_{Y}\rightarrow \infty ,$
which shows that the non optimal probability density functional is
appropriately `close' to the optimal functional. The accumulated execution
error associated with the non optimal functional is estimated as
\begin{equation}
\min_{\eta }\sqrt{\int_{T}^{T+R}\mathrm{var}\big[x(t))\big]dt}=O\bigg(\frac{1%
}{M_{Y}^{1/2-\alpha }}\bigg)  \label{error_est}
\end{equation}%
and, in this way, optimal performance of control, with sub-Poisson noise,
can be seen to become precise as $M_{Y}$ is made large. By contrast, if $%
\alpha \geq 0.5$, the accumulated execution error is always greater than
some positive constant.

To gain an intuitive understanding of why the effects of noise are
eliminated for $\alpha<0.5$ we discretise the time $t$ into small bins of
identical size $\Delta t$. Using the `noiseless control' $\hat{u}_{i}(t)$,
we divide the time bin $\left[ t,t+\Delta t\right] $ into two complementary
intervals: $\left[ t,t+|\hat{u}(t)|\Delta t/M_{Y}\right] $ and $\left[ t+|%
\hat{u}(t)|\Delta t/M_{Y},t+\Delta t\right] $, and assign $%
\lambda_{i}=\sigma(t)M_{Y}$ for the first interval and $\lambda_{i}=0$ for
the second. When $\Delta t\rightarrow0$ the effect of the control signal $%
\lambda_{i}(t)$ on the system approaches that of $\hat{u}_{i}(t)$, although $%
\lambda_{i}(t)$ and $\hat{u}_{i}(t)$ are quite different. The variance of
the noise in the first interval is $\kappa _{i}M_{Y}^{2\alpha}$ and is $0$
in the second. Hence, the overall noise effect of the bin is $\sigma_{i}^{2}=%
\frac{\kappa_{i}|\hat{u}_{i}(t)|}{M_{Y}}\cdot M_{Y}^{2\alpha}=\kappa_{i}|%
\hat{u}_{i}(t)|M_{Y}^{2\alpha-1}$. Remarkably, this tends to zero as $%
M_{Y}\rightarrow\infty$ \textit{if} $\alpha<1/2$ (i.e., for sub-Poisson
noise). The discretisation presented may be regarded as a formal stochastic
realisation of the probability density functional (Young measure) adopted.
The interpretation above can be verified in a rigorous mathematical way. See
Appendix B for details.

\section{Application and Example}

Let us now consider an application of this work: the control of
straight-trajectory arm movement, which has been widely studied \cite%
{Win,Sim,Tan,Ike} and applied to robotic control. The dynamics of such
structures are often formalised in terms of coordinate transformations.
Nonlinearity arises from the geometry of the joints. The change in spatial
location of the hand that results from bending the elbow depends not only on
the amplitude of the elbow movement, but also on the state of the shoulder
joint.

For simplicity, we ignore gravity and viscous forces, and only consider the
movement of a hand on a horizontal plane in the absence of friction. Let $%
\theta _{1}$ denote the angle between the upper arm and horizontal
direction, and $\theta _{2}$ be the angle between the forearm and upper arm
(Fig. \ref{arms}). The relation between the position of hand $[x_{1},x_{2}]$
and the angles $[\theta _{1},\theta _{2}]$ is
\begin{eqnarray*}
\theta _{1} &=&\arctan (x_{2}/x_{1})-\arctan (l_{2}\sin \theta
_{2}/(l_{1}+l_{2}\cos \theta _{2})) \\
\theta _{2} &=&\arccos
[(x_{1}^{2}+x_{2}^{2}-l_{1}^{2}-l_{2}^{2})/(2l_{1}l_{2})],
\end{eqnarray*}%
where $l_{1,2}$ are moments of inertia with respect to the center of
mass, for the upper arm and forearm. When moving a hand between two
points, a human maneuvers their arm so as to make the hand move in roughly a
straight line between the end points. We use this to motivate the model by
applying geostatics theory \cite{Win}. This implies that the arm satisfies
an Euler-Lagrange equation, which can be described as the following
nonlinear two-dimensional system of differential equations:%
\begin{eqnarray}
&&N(\theta _{1},\theta _{2})\left[
\begin{array}{c}
\ddot{\theta}_{1} \\
\ddot{\theta _{2}}%
\end{array}%
\right] +C(\theta _{1},\theta _{2},\dot{\theta}_{1},\dot{\theta}_{2})\left[
\begin{array}{c}
\dot{\theta}_{1} \\
\dot{\theta}_{2}%
\end{array}%
\right] =\gamma _{0}\left[
\begin{array}{c}
Q_{1} \\
Q_{2}%
\end{array}%
\right] ,  \nonumber \\
&&\theta _{1}(0)=-\frac{\pi }{2},~\theta _{2}(0)=\frac{\pi }{2},~\dot{\theta}%
_{1}(0)=\dot{\theta}_{2}(0)=0.  \label{AM}
\end{eqnarray}%
In these equations
\begin{eqnarray*}
N &=&\left[
\begin{array}{ll}
\begin{array}{l}
I_{1}+m_{1}r_{1}^{2}+m_{2}l_{1}^{2} \\
+I_{2}+m_{2}r_{2}^{2}+2k\cos \theta _{2}%
\end{array}
& I_{2}+m_{2}r_{2}^{2}+k\cos \theta _{2} \\
I_{2}+m_{2}r_{2}^{2}+k\cos \theta _{2} & I_{2}+m_{2}r_{2}^{2}%
\end{array}%
\right] , \\
C &=&k\sin \theta _{2}\left[
\begin{array}{ll}
\dot{\theta}_{2} & \dot{\theta}_{1}+\dot{\theta}_{2} \\
\dot{\theta}_{1} & 0%
\end{array}%
\right] ,~Q_{i}=\lambda _{i}(t)+\kappa _{0}|\lambda _{i}(t)|^{\alpha }\frac{%
dW_{i}}{dt},
\end{eqnarray*}%
where {$m_{i}$, $l_{i}$, and $I_{i}$ are, respectively the mass, length, and
moment of inertia with respect to the center of mass for the $i$'th part of
the system and $i=1$ ($i=2$) denotes the upper arm (forearm), $r_{1,2}$ are
the lengths of the upper- and fore-arms, and $\gamma _{0}$ is the scale
parameter of the force. Additionally, $k=m_{2}l_{1}r_{2}$, while } $\lambda
_{1,2}(t)$ are the \textit{means} of two torques $Q_{1,2}(t)$, which are
motor commands to the joints. The torques are accompanied by
signal-dependent noises. All other quantities are fixed parameters. See \cite%
{Win} for the full details of the model. The values of the parameters we
pick here are listed in Table \ref{parameters}.

For this example, we shall aim to control the hand such that it starts at $t=0$,
with the initial condition of (\ref{AM}), reaches the target at coordinates $%
H=[H_{1},H_{2}]$ at time $t=T$, and then stays at this target for a time
interval of $R$. We use the minimum variance principle to determine the
optimal task, which is more advantageous than other optimisation criteria to
control a robot arm \cite{Win,Ike}. Let $[x_{1}(t),x_{2}(t)]$ be the
Cartesian coordinates of the hand that follow from the angles $[\theta
_{1}(t),\theta_{2}(t)]$. The minimum variance principle determines $%
\min_{\lambda_{1},\lambda_{2}}\int_{T}^{T+R}[\mathrm{var}(x_{1}(t))+\mathrm{%
var}(x_{2}(t))]dt$, subject to the constraint that $\mathbf{E}%
[x_{1}(t),x_{2}(t)]=[H_{1},H_{2}] $ for $T\leq t\leq T+R$, with $%
-M_{Y}\leq\lambda_{i}\leq M_{Y}$. Despite not being in possession of an
explicit analytic solution, we can conclude that if $\alpha\geq0.5$, the
optimisation problem results from the unique minimum to the Hamiltonian
integrand and hence yields $\lambda_{1}(t)$ and $\lambda _{2}(t)$ which are
\textit{ordinary functions}. However, if $\alpha<0.5$, the optimal solution
of the optimisation problem follows from a probability density functional
analogous to Eq. (\ref{functional density}) (i.e., a Young measure over $%
\lambda_{i}\in\{-M_{Y},0,M_{Y}\}$). Thus, we can relax the optimisation
problem via Young measure as follows:
\begin{eqnarray*}
Q_{i}=\int_{-M_{Y}}^{M_{Y}} \bigg(\xi_{i}+\kappa_{0}|\xi_{i}(t)|^{%
\alpha}dW_{i}/dt\bigg)\cdot\eta_{i,t}(d\xi),~i=1,2,
\end{eqnarray*}
and
\begin{eqnarray}
\left\{%
\begin{array}{ll}
\min_{\eta_{1,2}(\cdot)} & \int_{T}^{T+R} [\mathrm{var}(x_{1}(t))^{2}+%
\mathrm{var}[x({2}(t))^{2}]dt \\
\mathrm{Subject~to} & \mathbf{E}[x_{1}(T),x_{2}(T)]=[H_{1},H_{2}],~t\in[T,T+R%
] \\
& \xi_{i}\in[-M_{Y},M_{Y}].%
\end{array}%
\right.  \label{RAMOP}
\end{eqnarray}

We used Euler's method to conduct numerical computations, with a time step
of $0.01$ msec in (\ref{AM}). This yields a dynamic programming problem (see
Methods). Fig. \ref{mean} shows the means of the optimal control signals $%
\bar{\lambda}_{1,2}(t)$ with $\alpha =0.25$ and $M_{Y}=20000$:
\[
\bar{\lambda}_{i}(t)=\int_{-M_{Y}}^{M_{Y}}\xi \eta _{i,t}(d\xi ),~i=1,2.
\]%
According to the form of the optimal Young measure, the optimal solution
should be
\[
\eta _{i,t}(ds)=\left\{
\begin{array}{ll}
\bigg[\frac{\bar{\lambda}_{i}(t)}{M_{Y}}\delta_{M_{Y}} (s)+(1-\frac{\bar{%
\lambda}_{i}(t)}{M_{Y}})\delta_{0} (s)\bigg]ds & \bar{\lambda}_{i}(t)>0 \\
\bigg[\frac{|\bar{\lambda}_{i}(t)|}{M_{Y}}\delta_{-M_{Y}} (s)+(1-\frac{|\bar{%
\lambda}_{i}(t)|}{M_{Y}})\delta_{0} (s)\bigg]ds & \bar{\lambda}_{i}(t)<0 \\
\delta_{0} (s)ds & \mathrm{otherwise}.%
\end{array}%
\right.
\]

It can be shown (derivation not given in this work) that in the absence of
the noise term, the arm can be accurately controlled to reach a given target
for any $T>0$. In this case, Fig. \ref{desire} shows the dynamics of the
angles, their velocities, and accelerators, in the controlled system,
removed noise. See, in comparison, the dynamical system with noise, whose
dynamics of the angles, velocities, and accelerations are illustrated in
Fig. \ref{fact}, and its dynamics are exactly the same as those in the case
with noise removed. However, the acceleration dynamics of a noisy dynamic
system appear discontinuous since the control signals, that have noises and
are added to the right-hand sides of the mechanical equations, are
discontinuous (noisy) in a numerical realisation. However, according to the
theory of stochastic differential equations \cite{SDE}, (\ref{AM}) has
continuous solution. Hence, these discontinuous acceleration dynamics lead
very smooth dynamics of velocities and angles, as shown in Fig. \ref{fact}.

Figs. \ref{figs1} (a) and (b) illustrate that the probability density
functional, for this problem, contains optimal control signals that are
similar to neural pulses. Despite the optimal solution not being an ordinary
function when $\alpha <0.5$, the trajectories of the angles $\theta _{1}$
and $\theta _{2}$ of the arm appear quite smooth, as shown in Fig. \ref{fact}
(a), and the target is reached very precisely if the value of $M_{Y}$ is
large. By comparison, when $\alpha >0.5$ the outcome has a standard
deviation between $4$ to $6$ cm, which may lead to a failure to reach
the target. A direct comparison between the execution error of the cases $%
\alpha =0.8(>0.5)$ and $\alpha =0.25(<0.5)$ is shown in the supplementary
movies (supplementary videos `Video S1' and `Video S2') of arm movements of
both cases. Our conclusion is that a Young measure optimal solution, in the
case of sub-Poisson control signals, can realize a precise control
performance even in the presence of noise. However, Poisson or Supra-Poisson
control signals cannot realise a precise control performance, despite the
existence of an explicit optimal solution in this case. Thus $\alpha <0.5$
significantly reduces execution error compared with $\alpha \geq 0.5$.

With different $T$ (the starting time of reaching the target) and $R$ (the
duration of reaching the target), under sub-Poisson noise, i.e., $\alpha <0.5
$, the system can be precisely controlled by optimal Young measure signals
with a sufficiently large $M_{Y}$. Since the target in the reachable region
of the arm, it implies that  the original differential system
of (\ref{AM}) with the noise removed can be controlled for any $T>0$ and $%
R>0$ \cite{Win,Sim}. According to the discussion in Appendix B (Theorem \ref%
{thm2}), the execution error can be arbitrarily small when $M_{Y}$ is
sufficiently large. However, for a smaller $T$, i.e., the more rapid the
control is, the larger means of the control signals will be. As for the
duration $R$, by picking the control signals as fixed values (zeros in this
example) such that the velocities keep zeros, the arm will stay at the
target for arbitrarily long or short. Similarly, with a large $M_{Y}$, the
error (variance) of staying at the target can be very small. To illustrate
these arguments, we take $T=100$ (msec) and $R=100$ (msec) for example (all
other parameters are the same as above). Fig. \ref{meanv1} shows that the
means of the optimal Young measure control signals before reaching the
target have larger amplitudes than those when $T=650$ (msec) and Fig. \ref%
{factv1} shows that the arm can be precisely controlled to reach and stay at
the target.

The movement error depends strongly on the value of the dispersion index, $%
\alpha$, and the bound of the control signal, $M_{Y}$. Fig. \ref{figs2}
indicates a quantitative difference in the execution error between the two
cases $\alpha<0.5$ and $\alpha\geq 0.5$, if $\alpha$ is close to (but less
than) $0.5$. The execution error can be appreciable unless a large $M_{Y}$
is used. For example if $\alpha=0.45$, as in Fig. \ref{figs2}, the square
root of the execution error is approximately $0.6$ cm when $M_{Y}=20000$.
From (\ref{error_est}), the error decreases as $M_{Y}$ increases, behaving
approximately as a power-law, as illustrated in the inner plot of Fig. \ref%
{figs2}. The logarithm of the square root of the execution error is found to
depend approximately linearly on the logarithm of $M_{Y}$ when $\alpha=0.25$%
, with a slope close to $-0.25$, in good agreement with the theoretical
estimate (\ref{error_est}).

We note that in a biological context, a set of neuronal pulse trains can
achieve precise control in the presence of noise. This could be a natural
way to approximately implement the probability density functional when $%
\alpha<0.5$. All other parameters are the same as above ($\alpha=0.25$). The
firing rates are illustrated in Fig. \ref{figs1} (a) and (b) and broadly
coincide with the probability density functional we have discussed. In
particular, at each time $t$, the probability $\eta_{i,t}$ can be
approximated by the fraction of the neurons that are firing, with the mean
firing rates equal the means of the control signals (see Methods). The
approximations of the components of the noisy control signals are shown in
Figs. \ref{figs3} (a) and (b) respectively. Fig.\ref{figs3} (c) and (d)
illustrate such an implementation of the optimal solution by neuronal pulse
trains. Using the pulse trains as control signals, we can realise precise
movement control. We enclose two videos `movieUP.avi' and 'movieDOWN.avi' to
demonstrate the efficiency of the control by pulse trains with two different
targets. As they show, the targets are precisely accessed by the arm. We
point out that the larger the ensemble is, the more precise the control
performance will be, because a large number of the neurons in an ensemble
can theoretically lead to a large $M_{Y}$ as we mentioned above, which
results in an improvement of the approximation of a Young measure and
decreases the execution error as stated in (\ref{error_est}).

We note that these kinds of patterns of pulse trains have been widely
reported in experiments, for example, the synchronous neural bursting
reported in \cite{Ross1}. This may provide a mathematical rationale for the
nervous system to adopt pulse-like signals to realise motor control.

\section{Conclusions}

In this paper, we have provided a general mathematical framework for
controlling a class of stochastic dynamical systems with random control
signals whose noisy variance can be regarded as a function of the signal
magnitude. If the dispersion index, $\alpha$, is $<0.5$, which is the case
when the control signal is sub-Poisson, an optimal solution of explicit
function does not exist but has to be replaced by a Young measure solution.
This parameterized measure can lead a precise control performance, where the
controlling error can become arbitrarily small. We have illustrated this
theoretical result via a widely-studied problem of arm movement control.

In the control problem of biological and robotic systems, large control
signals are needed for rapid movement control \cite{Ric}. When noise occurs,
this will cause imprecision in the control performance. As pointed out in
\cite{Sim,Tod,Kit,Mus,Chh}, a trade-off should be considered when conducting
rapid control with noises. In this paper, we still use a "large" control
signal but with different contexts. With sub-Poisson noises, we proved that
a sufficiently large $M_{Y}$, i.e., a sufficiently large region of the
control signal values, can lead precise control performance. Hence, a large
region of control signal values plays a crucial role in realising precise
control in noisy environments, for both "slow" and "rapid" movement control.
In numerical examples, the larger $M_{Y}$ we pick, the smaller control error
will be, as shown in the inset plot of Fig. \ref{figs2} as well as (\ref%
{error_est}) (Theorem \ref{thm2} in Appendix B).

Implementation of the Young measure approach in biological control appears
to be a natural way to achieve precise execution error in the presence of
sub-Poisson noise. In particular, in the neural-motor control example
illustrated above, the optimal solution in the case of $\alpha <0.5$ is
quite interesting. Assume we have an ensemble of neurons which fire pulses
synchronously within a sequence of non overlapping time windows, as depicted
in Figs. \ref{figs1} C and D. We see that the firing neurons yield control
signals which are very close, in form, to the type of Young measure
solution. This conclusion may provide a mathematical rationale for the
nervous system why to adopt pulse-like trains to realise motor control.
Additionally, we point out that, our approach may have significant
ramifications in other fields, including robot motor control and sparse functional estimation, which are issues of our future research.





\section*{Methods}

\textbf{Numerical methods for the optimisation solution.} We used Euler's
method to conduct numerical computations, with a time step of $\Delta t=0.01$
msec in (\ref{AM}) with $\alpha<0.5$. This yields a dynamic programming
problem. First, we divide the time domain $[0,T]$ into small time bins with
a small size $\Delta t$. Then, we regard the process $\eta_{i,t}$ in each
time bin $[n\Delta t,(n+1)\Delta t]$ as a static measure variable. Thus, the
solution reduces to finding two series of nonnegative parameters $\mu_{i,n}$
and $\nu_{i,n}$ with $\mu_{i,n}+\nu_{i,n}\leq1$ and $\mu_{i,n}\nu_{i,n}=0$
such that
\[
\lambda_{i}(t)=\left\{
\begin{array}{ll}
M_{Y} & n\Delta t\leq t<(n+\mu_{i,n})\Delta t \\
-M_{Y} & (n+\mu_{i,n})\Delta t\leq t<(n+\mu_{i,n}+\nu_{i,n})\Delta t \\
0 & (n+\mu_{i,n}+\nu_{i,n})\Delta t\leq t<(n+1)\Delta t.%
\end{array}
\right. .
\]
The approximate solution of the optimisation problem requires nonnegative $%
\mu_{i,n}$ and $\nu_{i,n}$ that minimise the final movement errors. We thus
have a dynamic programming problem. We should point out that in the
literature, a similar method was proposed to solve the optimisation problem
in a discrete system with control signals taking only two values \cite%
{Ike,Aih,Aih1}. Thus, the dynamical system (\ref{general}) becomes the
following difference equations via the Euler method:
\begin{eqnarray*}
&&x_{i}(k+1)=x_{i}(k)+\Delta t\bigg\{a_{i}(x(k),t_{k})+%
\sum_{j=1}^{m}b_{ij}(x(k),t_{k})[\mu _{j,k}-\nu _{j,k}]M_{Y}\bigg\} \\
&&+\sum_{j=1}^{m}b_{ij}(x(k),t_{k})\kappa _{j}\sqrt{\Delta t}(\mu _{j,k}+\nu
_{j,k})M_{Y}^{\alpha }\nu _{j},~k=0,1,2,\cdots ,
\end{eqnarray*}%
where $\nu _{j}$, $j=1,\ldots ,m$, are independent standard Gaussian random
variables. We can derive difference equations for the expectations and
variances of $x(k)$, by ignoring the higher order terms with respect to $%
\Delta t$:
\begin{eqnarray*}
&&\mathbf{E}(x_{i}(k+1))=\mathbf{E}(x_{i}(k))+\Delta t\bigg\{\mathbf{E}%
[a(x(k),t_{k})] \\
&&+\sum_{j=1}^{m}\mathbf{E}[b_{ij}(x(k),t_{k})][\mu _{j,k}-\nu _{j,k}]M_{Y}%
\bigg\} \\
&&\mathrm{cov}(x_{i}(k+1),x_{i^{\prime }}(k+1))=\mathrm{cov}%
(x_{i}(k),x_{i^{\prime }}(k))+\Delta t\mathrm{cov}\bigg(x_{i^{\prime }}(k),%
\big\{a_{i}(x(k),t_{k}) \\
&&+\sum_{j=1}^{m}b_{ij}(x(k),t_{k})[\mu _{j,k}-\nu _{j,k}]M_{Y}\big\}\bigg)
\\
&&+\Delta t\mathrm{cov}\bigg(x_{i}(k),\big\{a_{i^{\prime
}}(x(k),t_{k})+\sum_{j=1}^{m}b_{i^{\prime }j}(x(k),t_{k})[\mu _{j,k}-\nu
_{j,k}]M_{Y}\big\}\bigg) \\
&&+\Delta t\sum_{j=1}^{m}\mathrm{cov}(b_{i^{\prime
}j}(x(k),t_{k}),b_{ij}(x(k),t_{k}))(\mu _{j,k}+\nu _{j,k})M_{Y}^{2\alpha }.
\end{eqnarray*}%
Thus, Eq. (\ref{ROP}) becomes the following discrete optimization problem:
\begin{eqnarray}
\left\{%
\begin{array}{ll}
\min_{\mu _{i,k},\nu _{i,k}} & \sum_{k}\mathrm{var}(\phi (x(k),t_{k})) \\
\mathrm{subject~to~} & \mathbf{E}(\phi
(x(k),t_{k}))=z(t_{k}),~\mu_{i,k}+\nu_{i,k}\le 1, \\
& \mu_{i,k}\ge0,~\nu_{i,k}\ge0%
\end{array}%
\right.  \label{dis}
\end{eqnarray}%
with $x(k)$ a Gaussian random vector with expectation $\mathbf{E}(x(k))$ and
covariance matrix $cov(x(k),x(k))$.

\textbf{Neuronal pulse trains approximating Young measure solution.} At each
time $t $, the measure $\eta_{t}^{\ast}$ can be approximated by the fraction
of the neuron ensemble that are firing. In detail, assuming that the means
of the optimal control signals are $R_{t,i}(\xi)$, $i=1,2$, and there are
one ensemble of excitatory neurons and another ensemble of inhibitory
neurons. A fraction of the neurons fire so that the mean firing rates
satisfy:
\[
\lambda_{i}^{\ast}(t)=\gamma\bigg\{\mathbf{E}[R_{i}^{ext}(t)]-\mathbf{E}%
[R_{i}^{inh}(t)]\bigg\},
\]
where $R_{i}^{ext}(t)$ and $R_{i}^{inh}(t)$ are the firing rates of the
excitatory and inhibitory neurons respectively, and $\gamma$ is a scalar
factor. In occurrence of sub-Poisson noise, the noisy control signals $%
u_{i}^{\ast}(t)=\lambda_{i}^{\ast}(t)+\zeta_{i}(t)$ are approximated by
\[
u_{i}^{\ast}(t)=\gamma\big[R_{i}^{ext}(t)-R_{i}^{inh}(t)\big].
\]
Both ensembles of neurons are imposed with baseline activities, which bound
the minimum firing rates away from zeros, given the spontaneous activities
of neurons when no explicit signal is transferred. A numerical approach
involves discretise time $t$ into small bins of identical size $\Delta t$.
The firing rates can be easily estimated by averaging the population
activities in a time bin. We have used $400$ neurons to control the system,
with two ensembles of neurons with equivalent numbers that approximate the
first and second components of control signal, respectively. Each neuron
ensemble have $200$ neurons with $100$ excitatory and $100$ inhibitory
neurons.


\begin{thebibliography}{99}
\bibitem{XX} See \cite{Roub} for the details of definitions of Fr\'{e}chet
differential, G\^{a}teaux equi-differential, and partial differential, which
are extended from those of functional/operator on function space.
\end{thebibliography}




\section*{Appendices}

\subsection*{Appendix A: Derivation of formula (\protect\ref{h})}

Let: $W$ be the time-varying p.d.f. $p(x,t)$ that is second-order
continuous-differentiable with respect to $x$ and $t$ that is embedded in
the Sobolev function space $W^{2,2}$; $W_{1}$ be the function space of $%
p(x,t_{0})$, regarding as a function with respect to $x$ with a fixed $t_{0}$%
; $W_{2}$ be the function space of $p(x_{0},t)$, regarding as a function
with respect to $t$ with a given $x_{0}$. The spaces $W_{1,2}$ can be
regarded being embedded in $W$. In addition, let: $\hat{W}$ be the function
space where the image $\mathcal{L}[W]$ is embedded; $L$ be the space of
linear operator $\mathcal{L}$, denoted above; $\mathscr L(Z,E)$ be the space
composed of bounded linear operator from linear space $Z$ to $E$; and $%
Z^{\ast }$ be the dual space of the linear space $Z$: $Z^{\ast }=\mathscr %
L(Z,\mathbb{R})$. Furthermore, let $\tilde{\mathcal{Y}}$ be the tangent
space of Young measure space $\mathcal{Y}$: $\tilde{Y}=\{\eta -\eta ^{\prime
}:~\eta ,\eta ^{\prime }\in \mathcal{Y}\}$. For simplicity, we do not
specify the spaces and just provide the formalistic algebras, and then the
following is similar to Chapter 4.3 in \cite{Roub} with appropriate
modifications.

Define
\begin{eqnarray}
\Phi(p)&=&\int_{T}^{T+R}\int_{\Xi} \|\phi(x,t)-z(t)\|^{2}p(x,t)dx dt
\nonumber \\
\Pi(p,\eta)&=&\bigg(\frac{\partial p}{\partial t}-(\mathcal{L}\cdot\eta)
\circ p,p(x,0)-p_{0}(x)\bigg)  \nonumber \\
J(p)&=&\int_{\Xi}\phi(x,t)p(x,t)dx-z(t).  \label{sym}
\end{eqnarray}
Thus, (\ref{ROP}) can be rewritten as:
\begin{eqnarray}
\left\{%
\begin{array}{ll}
\min_{\eta} & \Phi(p) \\
\mathrm{subject~to} & \Pi(p,\eta)=0,~J(p)=0.%
\end{array}%
\right.  \label{OPT}
\end{eqnarray}
The G\^{a}teaux differentials of these maps with respect to $p(x,t)$,
denoted by $\nabla_{p}\cdot$, are:
\begin{eqnarray*}
(\nabla_{p}\Phi)\circ(\hat{p}-p)&=&\int_{T}^{T+R}\int_{\Xi}\|\phi(x,t)-z(t)%
\|^{2} [\hat{p}(x,t)-p(x,t)]dx dt \\
(\nabla_{p}{\Pi})\circ(\hat{p}-p)&=&\bigg(\frac{\partial(\hat{p}-p)}{%
\partial t} -(\mathcal{L}\cdot\eta)\circ(\hat{p}-p),\hat{p}(x,0)-p(x,0)\bigg)
\\
(\nabla_{p}J)\circ(\hat{p}-p)&=&\int_{\Xi}\phi(x,t)[\hat{p}(x,t)-p(x,t)]dx \\
\end{eqnarray*}
for two time-varying p.d.f. $\hat{p},p\in W$. Here, $\nabla_{p}\Phi\in W^{*}$%
, $\nabla_{p}{\Pi}\in\mathscr L(W, \hat{W}\times W_{1})$, $\nabla_{p} J\in%
\mathscr L(W,W_{2})$. And, the differentials of these maps with respect to
the Young measure $\eta$ are:
\begin{eqnarray*}
(\nabla_{\eta}\Phi)\cdot(\hat{\eta}-\eta)&=&0 \\
(\nabla_{\eta}\Pi)\cdot(\hat{\eta}-\eta)&=& \bigg((\mathcal{L}\circ p)\cdot (%
\hat{\eta}-\eta),0\bigg) \\
(\nabla_{\eta}J)\cdot(\hat{\eta}-\eta)&=&0 \\
\end{eqnarray*}
for two Young measures $\hat{\eta},\eta\in\mathcal{Y}$. Here, $%
\nabla_{\eta}\Phi\in\tilde{\mathcal{Y}}^{*}$, $\nabla_{\eta}\Pi\in\mathscr L(%
\tilde{\mathcal{Y}},\hat{W}\times W_{1})$, and $\nabla_{\eta}J\in\mathscr L(%
\tilde{\mathcal{Y}},W_{2})$.

Then, we are in the position to derive the result of (\ref{h}) by the
following theorem, as a consequence from Theorem 4.1.17 in \cite{Roub}.

\begin{theorem}
\label{thm1} $\Phi:~W\to W^{*}$, $\Pi:~W\times \mathcal{Y}\to\mathbb{R}$ and
$J:~W\to W_{2}$ as defined in (\ref{sym}). Assume that: (1). the trajectory
of $x(t)$ in (\ref{Ito}) is bounded almost surely; (2). $a(x,t)$, $b(x,t)$
and $\phi(x,t)$ are $C^{2}$ with respect to $(x,t)$. 
Let $(\eta^{\ast},p^{\ast})$ be the optimal solution of (\ref{ROP}). Then,
there are some $\lambda_{1}\in\mathscr L(\hat{W}\times W_{1},W^{*}) $, $%
\lambda_{2}=[\lambda_{21},\lambda_{22}]^{\top}$ with $\lambda_{21}\in {%
\mathscr L}(\hat{W},W_{2})$, and $\lambda_{22}\in{\mathscr L}(W_{1},,W_{2})$%
, such that
\begin{eqnarray}
\lambda_{1}\circ\nabla_{p}\Pi(p^{\ast},\eta^{\ast})
=\nabla_{p}\Phi(p^{\ast}),~\lambda_{2}\circ\nabla_{p}\Pi(p^{\ast},\eta^{%
\ast})=\nabla_{p}J(p^{\ast}),  \label{cond2}
\end{eqnarray}
and the abstract maximum principle
\begin{equation}
\eta _{t}^{\ast}\bigg\{\mathrm{minima~of~}h(t,\xi )\mathrm{~w.r.t}~\xi %
\bigg\}=1,~\forall ~t\in \lbrack 0,T+R].  \label{amp}
\end{equation}
holds with "abstract Hamiltonian":
\begin{eqnarray}
h(t,\xi)&=&-\int_{\Xi }p^{\ast }(x,t)\bigg\{\sum_{i=1}^{n}A_{i}(x,t,\xi )%
\frac{\partial \mu _{1}}{\partial x_{i}}  \nonumber \\
&&+\frac{1}{2}\sum_{i,j=1}^{n}[B(x,t,\xi )B(x,t,\xi )^{\top }]_{ij}\frac{%
\partial ^{2}\mu _{1}}{\partial x_{i}\partial x_{j}}\bigg\}dx.  \label{h1}
\end{eqnarray}
\end{theorem}

\begin{proof}
Under the conditions in this theorem, we can conclude that the Fokker-Planck
equation has a unique solution $p(\eta)$ that is continuously dependent of $%
\eta$ from theory of stochastic differential equation \cite{SDE}; $%
\Pi(\cdot,\eta):W\to W^{*}$ is Fr\'{e}chet differentiable at $p=\pi(\eta)$
because $x(t)$ is assumed to be almost surely bounded; $\Pi(p,\cdot):%
\mathcal{Y}\to W^{*}$ and $J$ (in fact $\nabla_{\eta}J=0$) is G\^{a}teaux
equi-differentiable around $p\in W$ because of $p\in W\subset W^{2,2}$ with $%
(x,t)$ bounded \cite{XX}; the partial differential $\nabla_{\eta}\Pi$ is
weak-continuous with respect to $\eta$ because it is linearly dependent of $%
\eta$.

In addition, from the existence and uniqueness of the Fokker-Planck
equation, $\nabla_{p}\Pi(p,\eta):W\to Im (\nabla_{p}\Pi(p,\eta))\subset \bar{%
W}\times W_{1}$ has a bounded inverse. This implies that the following\emph{%
adjoint equation}
\[
\mu \circ \nabla _{p}\Pi =\nabla _{p}\Phi +c(t)\circ \nabla _{p}J,
\]%
has a solution for $p=p(\eta)$, denoted by $\mu$. Let $\mu =[\mu _{1},\mu
_{2}]$, which should be solution of the following equation
\begin{equation}
\left\{
\begin{array}{l}
\frac{\partial \mu _{1}}{\partial t}+(\mathcal{L}^{\ast }\cdot \eta ^{\ast
})\circ \mu _{1}=-\Vert \phi (x,t)-z(t)\Vert ^{2}-c(t)\phi (x,t) \\
\mu _{2}(x)=\mu _{1}(x,0) \\
\mu _{1}(x,T+R)=0 \qquad t\in \lbrack 0,T+R],~x\in \Xi ,%
\end{array}%
\right.
\end{equation}%
with the dual operator $\mathcal{L}^{\ast }$ of $\mathcal{L}$ (the operator
in the back-forward Kolmogorov equation), still dependent of $(x,t)$ and the
value of $\lambda (t)$ (namely $\xi$ in Young measure):
\[
\mathcal{L}^{\ast }\circ q=\sum_{i=1}^{n}a_{i}(x,t,\xi )\frac{\partial q}{%
\partial x_{i}}+\frac{1}{2}\sum_{i,j=1}^{n}[B(x,t,\xi )B^{\top }(x,t,\xi
)]_{ij}\frac{\partial ^{2}q}{\partial x_{i}x_{j}}.
\]%
We pick $\lambda_{1,2}(t)$ with $\mu _{1}=\lambda _{1}+c(t)\lambda _{21}$
and $\mu _{2}=\lambda _{1}+c(t)\lambda _{22}$. So, $\lambda_{1,2}$ should
satisfy equation (\ref{cond2}). In fact, $\lambda_{1,2}$ can be regarded as
functions ( or generalized functions) with respect to $(x,t)$.

Thus, the conditions of Lemma 1.3.16 in \cite{Roub} can be verified, which
implies that the gradients of the maps $\Phi$ and $J$ with respect to $\eta$%
, by regarding $p=p(\eta)$ from $\Pi(p,\eta)=0$, as follows:
\begin{eqnarray*}
\partial\Phi=\nabla_{\eta}\Phi-\lambda_{1}\circ\nabla_{\eta}\Pi, ~\partial
J=\nabla_{\eta}J-\lambda_{2}\circ\nabla_{\eta}\Pi.
\end{eqnarray*}

From the abstract Hamilton minimum principle (Theorem 4.1.17 in \cite{Roub}%
), applied to each solution of (\ref{ROP}), denoted by $\eta ^{\ast }$,
there exists a nonzero function $c(t) $ such that
\begin{equation}
H(\tilde{\eta})=\partial \Phi (\eta ^{\ast })\cdot \tilde{\eta}+\langle
c(t),\partial J(\eta ^{\ast })\cdot \tilde{\eta}\rangle ,~\forall ~\tilde{%
\eta}\in \mathcal{Y}  \label{cond1}
\end{equation}%
is an 'abstract Hamiltonian', with respect to $\tilde{\eta}$. With the
definitions of $\lambda_{1,2}$, (\ref{cond1}) becomes
\[
H(\tilde{\eta})=\nabla _{\eta }\Phi \cdot \tilde{\eta}+c\nabla _{\eta
}J\cdot \tilde{\eta}-\langle \mu ,\nabla _{\eta }\Pi \cdot \tilde{\eta}%
\rangle =-\langle \mu ,\nabla _{\eta }\Pi \cdot \tilde{\eta}\rangle ,
\]%
owing to $\nabla _{\eta }\Phi =\nabla _{\eta }J=0$.

By specifying $\mu$ with $\lambda_{1,2}$, we have
\begin{eqnarray*}
H(\tilde{\eta}) &=&-\langle \mu ,\nabla _{\eta }\Pi \cdot \tilde{\eta}%
\rangle =-\langle \mu _{1},[\mathcal{L}\circ p^{\ast }]\cdot \tilde{\eta}%
\rangle =-\langle p^{\ast },[\mathcal{L}^{\ast }\circ \mu _{1}]\cdot \tilde{%
\eta}\rangle \\
&=&-\int_{0}^{T+R}\int_{\Xi }\int_{\Omega }p^{\ast }\bigg\{%
\sum_{i=1}^{n}a_{i}(x,t,\xi )\frac{\partial \mu _{1}}{\partial x_{i}} \\
&&+\frac{1}{2}\sum_{i,j=1}^{n}[B(x,t,\xi )B(x,t,\xi )^{\top }]_{ij}\frac{%
\partial ^{2}\mu _{1}}{\partial x_{i}\partial x_{j}}\bigg\}\tilde{\eta}%
_{t}(d\xi )dxdt,
\end{eqnarray*}%
where $p^{\ast }$ stands for the time-varying density corresponding to the
optimal Young measure solution $\eta ^{\ast }$. From this, letting $%
\xi=\lambda$, we have the "abstract Hamiltonian" in the form of (\ref{h1})
as the Hamiltonian integrand of $H(\cdot)$.

The Hamiltonian abstract minimum (maximum) principle indicates the optimal
Young measure $\eta _{t}^{\ast }$ is only concentrated at the minimum points
of $h(t,\xi )$ with respect to $\xi $ for each $t$, namely. That is, (\ref%
{amp}) holds. This completes the proof.
\end{proof}

From this theorem, since the variances depend on the magnitude of the signal
as described in (\ref{sigma i}), removing the terms without $\xi$, it is
equivalent to look at the minima of $\hat{h}(t,\xi )$ in the form of
\begin{equation}
\hat{h}(t,\xi ) =\sum_{i=1}^{m}h_{i}(t,\xi _{i}), \quad h_{i}(t,\xi _{i})
=g_{i}(t)|\xi _{i}|^{2\alpha }-f_{i}(t)\xi _{i}  \label{Hamilton}
\end{equation}%
instead of $h(t,\xi )$, where
\begin{eqnarray*}
f_{i}(t) &=&\sum_{j=1}^{m}\int_{\Xi }p^{\ast }(x,t)b_{ji}(x,t)\frac{\partial
\mu _{1}}{\partial x_{i}}dx, \\
g_{i}(t) &=&-\frac{\kappa _{i}^{2}}{2}\sum_{j,k=1}^{m}\int_{\Xi }p^{\ast
}(x,t)b_{ji}(x,t)b_{ki}(x,t)\frac{\partial ^{2}\mu _{1}}{\partial
x_{j}\partial x_{k}}dx.
\end{eqnarray*}%
This gives formula (\ref{h}).

\subsection*{Appendix B: Derivation of precise control performance (\protect
\ref{error_est})}

The control performance inequality (\ref{error_est}) can be derived from the
following theorem.

\begin{theorem}
\label{thm2} Let $\hat{x}$ be the solution of equation (\ref{ds}) and $%
0<\alpha<0.5$. Assume that there are a positive measurable function $%
\kappa(t)$ and a positive constant $C_{1}$ such that $\Vert
A(x,t,u)-A(y,t,u)\Vert\leq\kappa(t)\Vert x-y\Vert$, $\Vert
B(x,t,u)\Vert^{2}\leq\kappa(t)\sum_{k=1}^{m}|u_{k}|^{2\alpha} $ and$%
|\phi(x,t)-\phi(y,t)|\leq C_{1}\Vert x-y\Vert$ hold for all $x,y\in\mathbb{R}%
^{n}$ and $t\geq 0$. Then, for any non-random initial value, namely, $x(0)=%
\mathbf{E}(x(0))$, with the non-optimal Young measure (\ref{functional
density1}), we have

\begin{enumerate}
\item $\int_{T}^{T+R}\Vert\mathbf{E}(x(t))-z(t)\Vert\to 0$;

\item $\min_{\eta}\sqrt{\int_{T}^{T+R}\mathrm{var}(x)dt}=O\bigg(\frac{1}{%
M_{Y}^{1/2-\alpha}}\bigg)$
\end{enumerate}

as $M_{Y}\to\infty$.
\end{theorem}

\begin{proof}
Comparing the differential equation of $x$, i.e. (\ref{Ito}), and that of $%
\hat{x}$, (\ref{ds}), we have $d(x-\hat{x})=[A(x,t,\hat{u})-A(\hat{x},t,\hat{%
u})]dt+B(x,t,\lambda(t))dW_{t}$. And, replacing $\lambda(t)$ with the Young
measure $\hat{\eta}_{t}(\cdot)$, in the form of (\ref{functional density1}),
from the conditions in this theorem, we have
\begin{eqnarray*}
& \mathbf{E}\Vert x(\tau)-\hat{x}(\tau)\Vert^{2}=\mathbf{E}\bigg\{\int
_{0}^{\tau}[A(x(t),t,\hat{u})-A(\hat{x}(t),t,\hat{u})]dt\bigg\}^{2} \\
& +\mathbf{E}\bigg\{\int_{0}^{\tau}\Vert B(x,t,\lambda)\Vert^{2}\cdot \hat{%
\eta}_{t}dt\bigg\} \\
& \leq\mathbf{E}\bigg\{\int_{0}^{\tau}\kappa(t)\Vert x(t)-\hat{x}(t)\Vert dt%
\bigg\}^{2}+\sum_{k=1}^{m}\int_{0}^{\tau}\kappa(t)|\lambda_{k}|^{2\alpha
}\cdot\hat{\eta}_{k,t}dt \\
& \leq\int_{0}^{\tau}\kappa^{2}(s)ds\int_{0}^{\tau}\mathbf{E}\Vert x(t)-%
\hat {x}(t)\Vert^{2}dt+\sum_{k=1}^{m}\int_{0}^{\tau}\frac{M^{2\alpha}_{Y}}{%
M_{Y}}\kappa(t)|\hat{u}_{k}(t)|dt,
\end{eqnarray*}
for any $\tau>0$. By using Gr\"{o}nwall's inequality, we have
\begin{eqnarray}
&& \int_{T}^{T+R}\mathbf{E}\Vert x(\tau)-\hat{x}(\tau)\Vert^{2}d\tau dt\leq
\int_{0}^{T+R}\mathbf{E}\Vert x(\tau)-\hat{x}(\tau)\Vert^{2}d\tau dt\le
\nonumber \\
&&\int_{0}^{T+R}\exp\bigg[\int_{t}^{T+R}\int_{0}^{s}\kappa^{2}(\tau)d\tau ds%
\bigg] \frac{1}{M^{1-2\alpha}_{Y}}\sum_{k=1}^{m}\kappa(t)|\hat{u}_{k}(t)|dt.
\label{errorest}
\end{eqnarray}
Noting that for $0<\alpha<0.5$, $\lim_{M_{Y}\rightarrow\infty}1/M_{Y}^{1-2%
\alpha}=0$ implies that $\int_{T}^{T+R}\mathbf{E}\Vert x(\tau)-\hat{x}%
(\tau)\Vert^{2}d\tau dt=O(1/M_{Y}^{1-2\alpha})$ as $M_{Y}$goes to infinity.
This proves the second item in this theorem.

In addition,
\begin{eqnarray*}
\int_{T}^{T+R}\Vert\mathbf{E}(x(t))-z(t)\Vert \leq \sqrt{R}\sqrt{\mathbf{E}%
\int_{T}^{T+R}\{\Vert x(t)-\hat{x}(t)\Vert^{2}dt\}}
\end{eqnarray*}
also approaches zero as $M_{Y}$ goes to infinity. This proves the first item
of the theorem. This completes the proof.
\end{proof}

Hence, as $M_{Y}$ goes to infinity, the non optimal solution (\ref%
{functional density1}) can asymptotically satisfy the constraint and the
error variance goes to zero as $M_{Y}$ goes to infinity. Therefore, the
performance error of the REAL optimal solution of the optimisation problem (%
\ref{ROP}) approaches zero as $M_{Y}\rightarrow\infty$ in the case of $%
\alpha<0.5$. Furthermore, we can conclude from (\ref{errorest}) that the
execution error, measured by the standard deviation, can be approximated as (%
\ref{error_est}).


\newpage

\begin{figure}[tbp]
\centering\includegraphics[width=.8\textwidth]{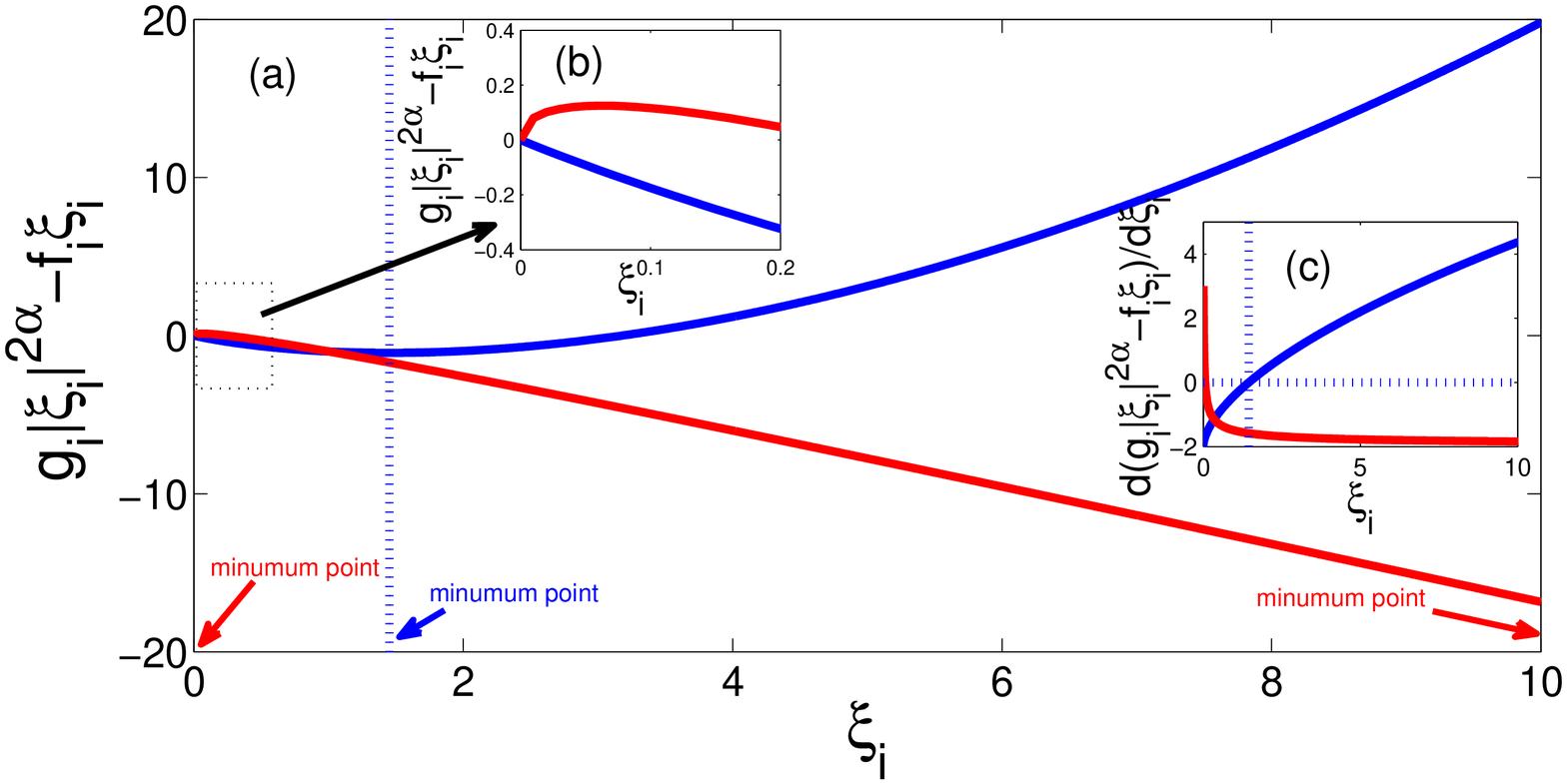}
\caption{Illustration of possible minimum points of the function $g_{i}(t)|%
\protect\xi_{i}|^{2\protect\alpha}-f_{i}(t)\protect\xi_{i}$ in $h_{i}(t,%
\protect\xi_{i}) $ with respect to the variable $\protect\xi_{i}$ with $%
g_{i}(t)=1$, $f_{i}(t)=2$ and $M_{Y}=10$ for $\protect\alpha=0.8>0.5$ (blue
curves) and $\protect\alpha=0.25<0.5$ (red curves): (a). the plots of $%
g_{i}(t)|\protect\xi_{i}|^{2\protect\alpha}-f_{i}(t)\protect\xi_{i} $ with
respect to $\protect\xi_{i}$ for $\protect\alpha=0.8$ (blue) and $\protect%
\alpha=0.25$ (red) and their mimimum points; (b) the inner plot of $g_{i}(t)|%
\protect\xi_{i}|^{2\protect\alpha}-f_{i}(t)\protect\xi_{i}$ for $\protect\xi%
_{i}\in[0,0.2]$ to show that the $\protect\xi_{i}=0$ does be a minimum point
for $\protect\alpha=0.25$ (red); (c). the plots of the derivatives of $%
g_{i}(t)|\protect\xi_{i}|^{2\protect\alpha}-f_{i}(t)\protect\xi_{i}$ with
respect to $\protect\xi_{i}$ for $\protect\alpha=0.8$ (blue) and $\protect%
\alpha=0.25$ (red). }
\label{al}
\end{figure}

\begin{table}[ht]
\caption{Summary of the possible minimum points of (\protect\ref{h}).}
\label{maximum points}\centering{%
\begin{tabular}{@{\vrule height 10.5pt depth4pt  width0pt}lcc}
& $g_{i}(t)>0$ & $g_{i}(t)<0$ \\ \hline
$\alpha>0.5$ & one point in $[-M_{Y},M_{Y}]$ & $\{M_{Y}\}$ or $\{-M_{Y} \}$
\\
$\alpha<0.5$ & $\{0,M_{Y}\}$ or $\{0,-M_{Y}\}$ & $\{M_{Y}\}$ or $\{-M_{Y} \}$
\\ \hline
\end{tabular}%
}
\end{table}

\begin{figure}[tbp]
\centering\includegraphics[width=.8\textwidth]{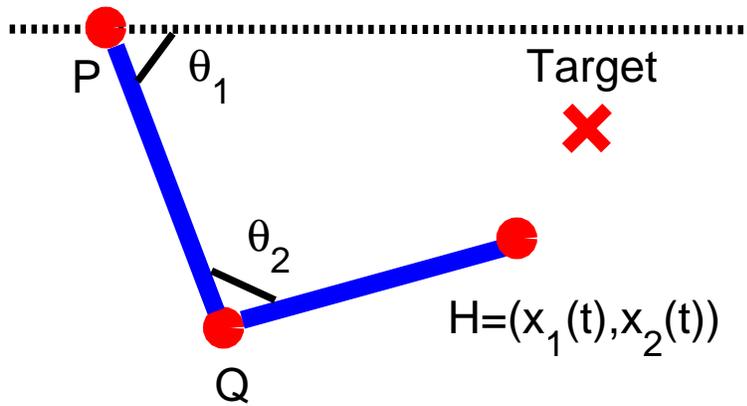}
\caption{Illustration of the arm control. Arm is composed of three points ($%
P $, $Q$, and $H$), where $P$ is fixed and others not, and two arms (upper
arm $PQ$ and the forearm $QH$). Button $H$ is to reach some given target
(red cross) by moving front- and back-arms. }
\label{arms}
\end{figure}

\begin{table}[ht]
\caption{Parameters.}
\label{parameters}{\centering%
\begin{tabular}{@{\vrule height 10.5pt depth4pt  width0pt}lc}
Parameters & Values \\ \hline
masses (of the inertia w.r.t the mass center) & $m_{1}=2.28~kg$, $%
m_{2}=1.31~kg$ \\
lengths (of the inertia w.r.t the mass center) & $l_{1}=0.305~m$, $%
l_{2}=0.254~m$ \\
moments (of the inertia w.r.t the mass center) & $I_{1}=0.022~kg\cdot m^2$; $%
I_{2}=0.0077~kg\cdot m^{2}$ \\
lengths of arms & $r_{1}=0.133~m$, $r_{2}=0.109~m$ \\
reach time & $T=650~ms$ (except in Figs. \ref{meanv1} and \ref{factv1}) \\
duration & $R=10~ms$ (except in Figs. \ref{meanv1} and \ref{factv1}) \\
target & $\theta_{1}(T)=-1$, $\theta_{2}(T)=\pi/2$ \\
scale parameter & $r_{0}=1$ \\
noise scale & $\kappa_{0}=1$ \\
bound of the control signal & $M_{Y}=20000$, except in Figs. \ref{al} and %
\ref{figs1} \\
& and the inset plot of Fig. \ref{figs2} \\
time step & $\Delta t=0.01~ms$ \\ \hline
\end{tabular}%
}
\end{table}

\begin{figure}[tbp]
\centering\includegraphics[width=.9\textwidth]{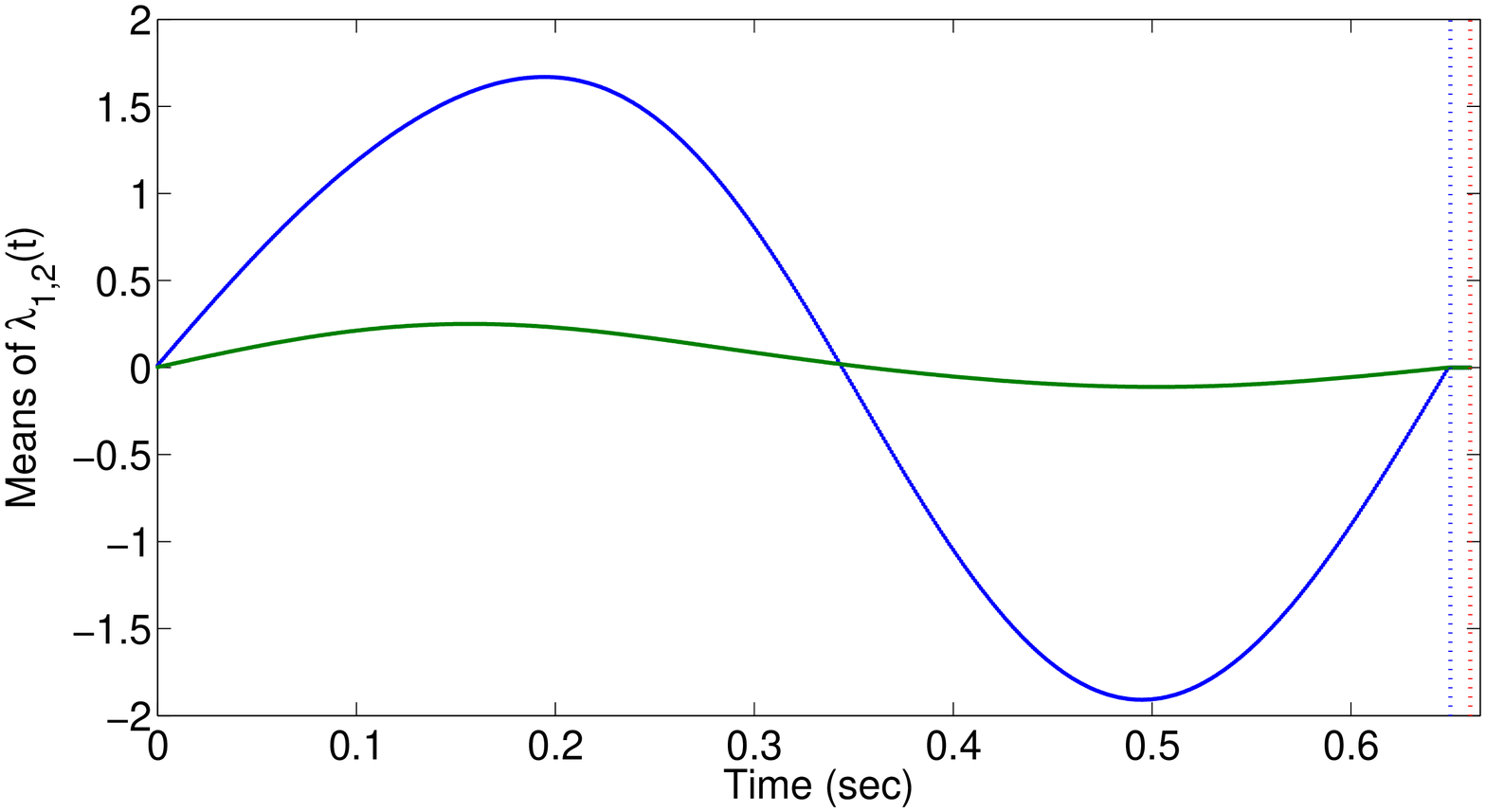}
\caption{The means of the optimal control signals $\protect\lambda_{1}(t)$
(blue solid) and $\protect\lambda_{2}$ (green solid) in the
straight-trajectory arm movement example with $\protect\alpha =0.25$ and $%
M_{Y}=20000$. The blue and red dash vertical lines stand for the start and
end time points of the duration of reaching the target respectively.}
\label{mean}
\end{figure}

\begin{figure}[tbp]
\centering\includegraphics[width=.9\textwidth]{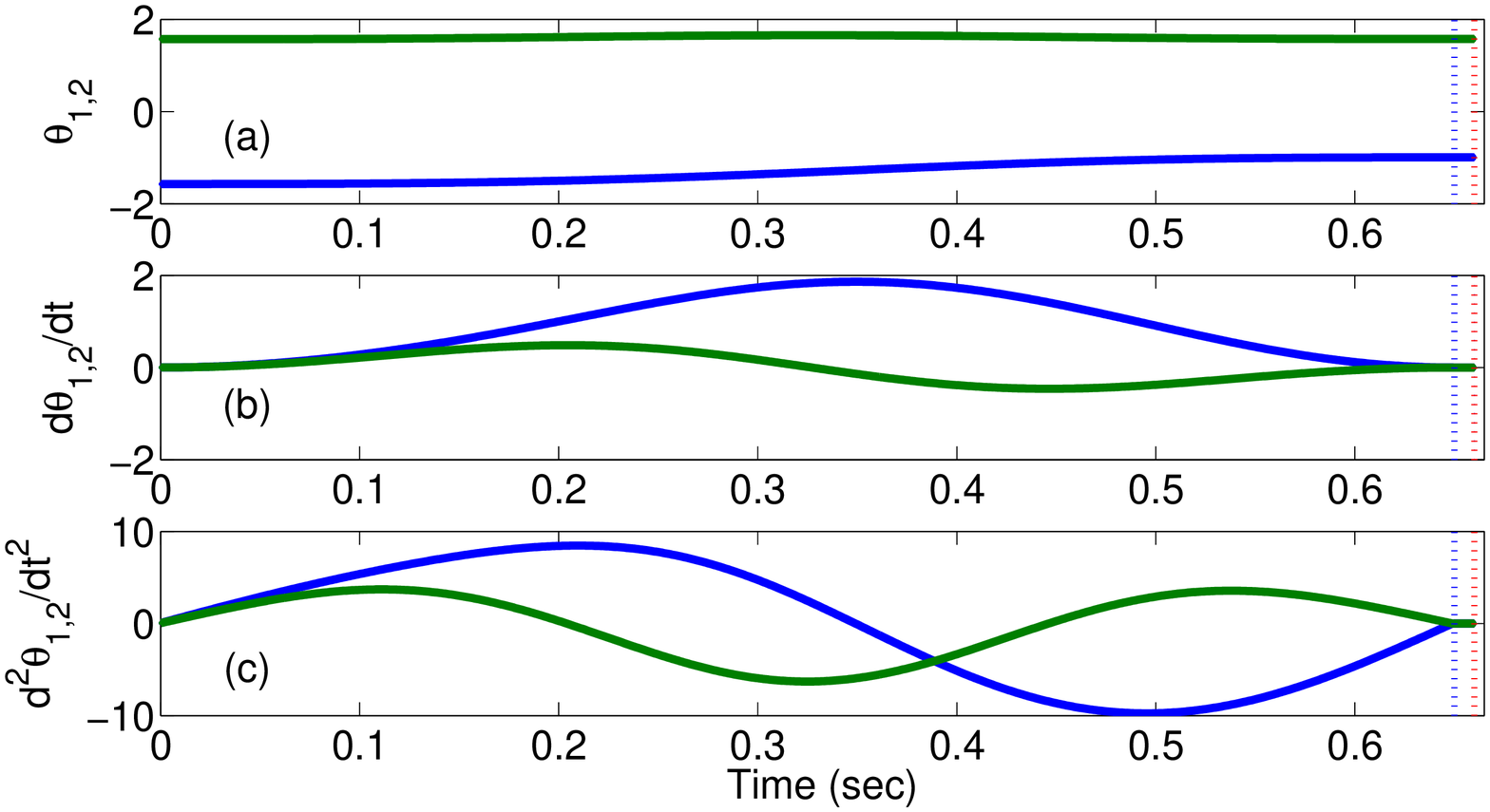}
\caption{Optimal control of straight-trajectory arm movement model with
parameters listed in Table \protect\ref{parameters}. Target is set by $%
\protect\theta_{1}(T)=-1$ and $\protect\theta_{2}(T)=\frac{\protect\pi}{2}$
but \emph{without noise}: The dynamics of the angles (a), the angle
velocities (b) and accelerations (c) (the blue solid curves for those of $%
\protect\theta _{1}$ and the green solid curves for $\protect\theta _{2}$).
The blue and red dash vertical lines stand for the start and end time points
of the duration of reaching the target.}
\label{desire}
\end{figure}

\begin{figure}[tbp]
\centering\includegraphics[width=.9\textwidth]{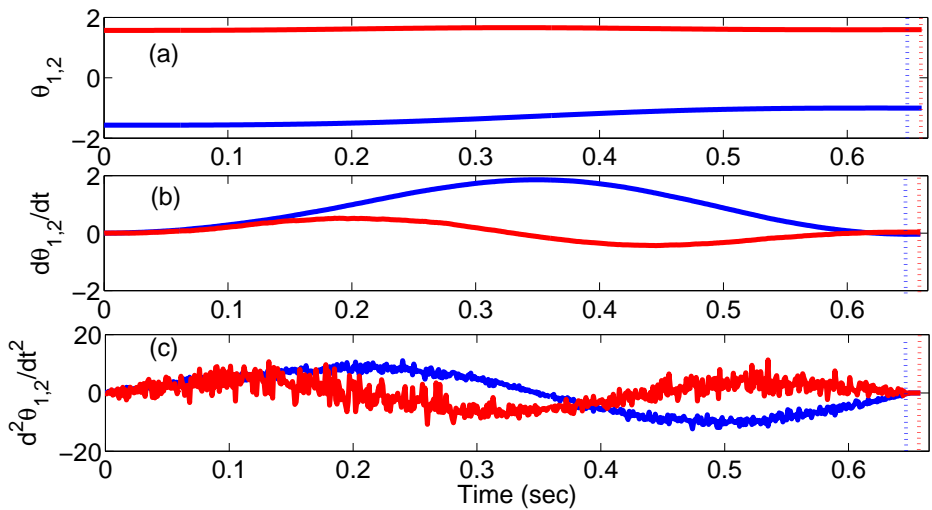}
\caption{Optimal control of straight-trajectory arm movement model \emph{%
with noise} and the same model parameters as in Fig. \protect\ref{desire},
and $\protect\alpha=0.25$, $M_{Y}=20000$: The dynamics of the angles (a),
the angle velocities (b) and accelerations (c) (the blue solid curves for
those of $\protect\theta_{1}$ and the red solid for $\protect\theta_{2}$).
The blue and red dash vertical lines stand for the start and end time points
of the duration of reaching the target.}
\label{fact}
\end{figure}

\begin{figure}[ptb]
\centering{\includegraphics[width=0.9
\textwidth]{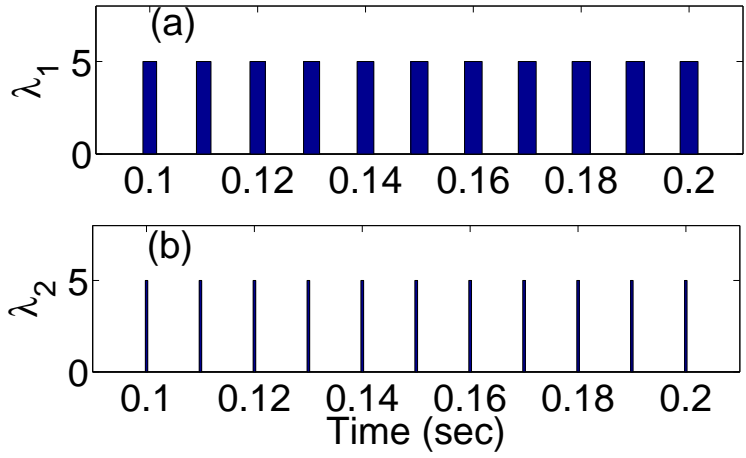}} 
\caption{Optimal control signal of Young measure of straight-trajectory arm
movement model with noise, illustrations for $\protect\lambda_{1}$ (a) and $%
\protect\lambda_{2}$ (b) in discrete-time way with $M_{Y}=5$, where width of
each bar stands for measure of $M_{Y}$ at each time bin. }
\label{figs1}
\end{figure}


\begin{figure}[tbp]
\centering\includegraphics[width=.9\textwidth]{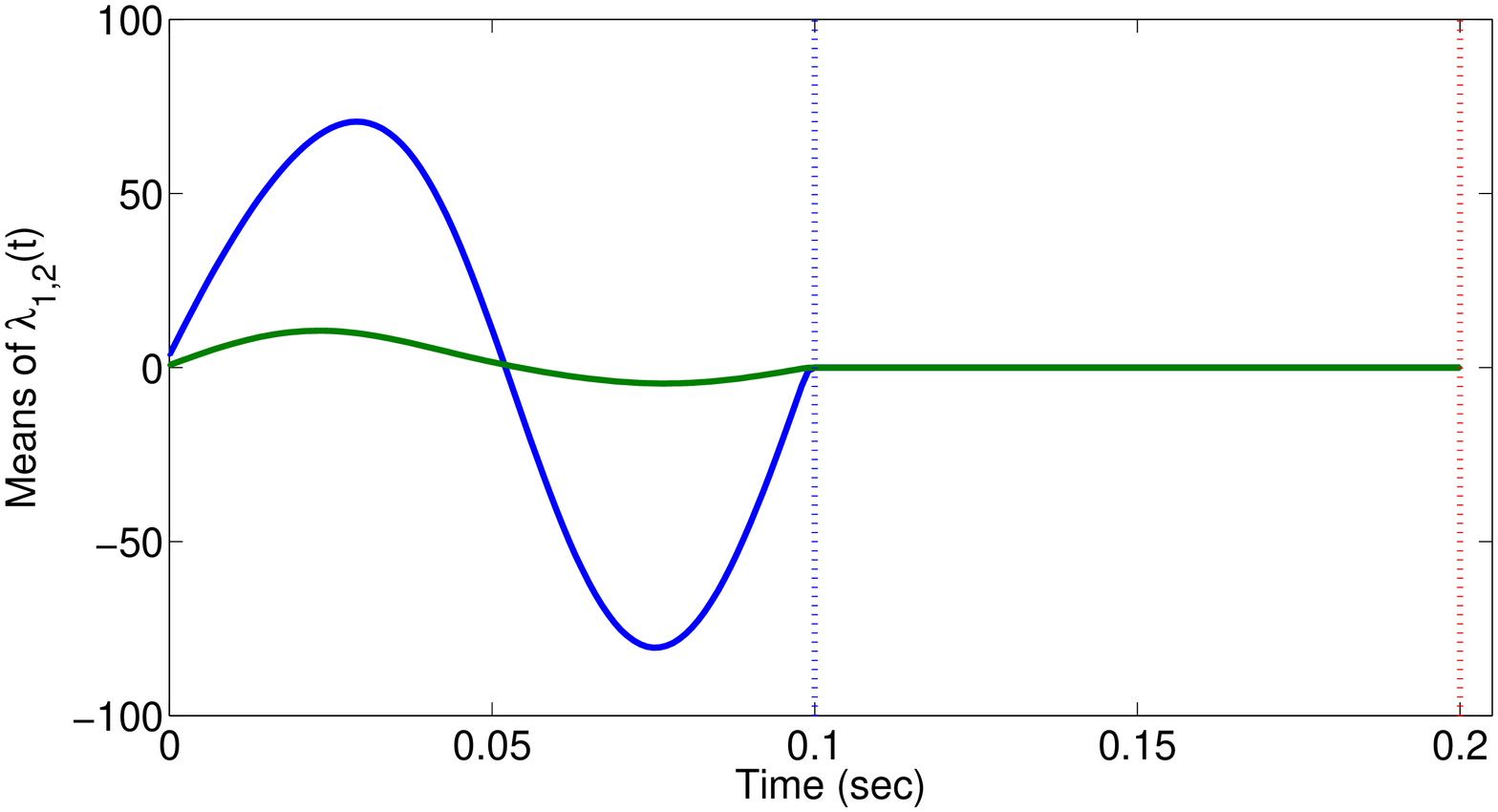}
\caption{Means of the optimal control signals $\protect\lambda _{1}(t)$
(blue solid) and $\protect\lambda _{2}(t)$ (green solid) in the
straight-trajectory arm movement example with $T=100$ (sec) and $R=100$
(msec). The blue and red dash vertical lines stand for the start and end
time points of the duration of reaching the target.}
\label{meanv1}
\end{figure}

\begin{figure}[tbp]
\centering\includegraphics[width=.9\textwidth]{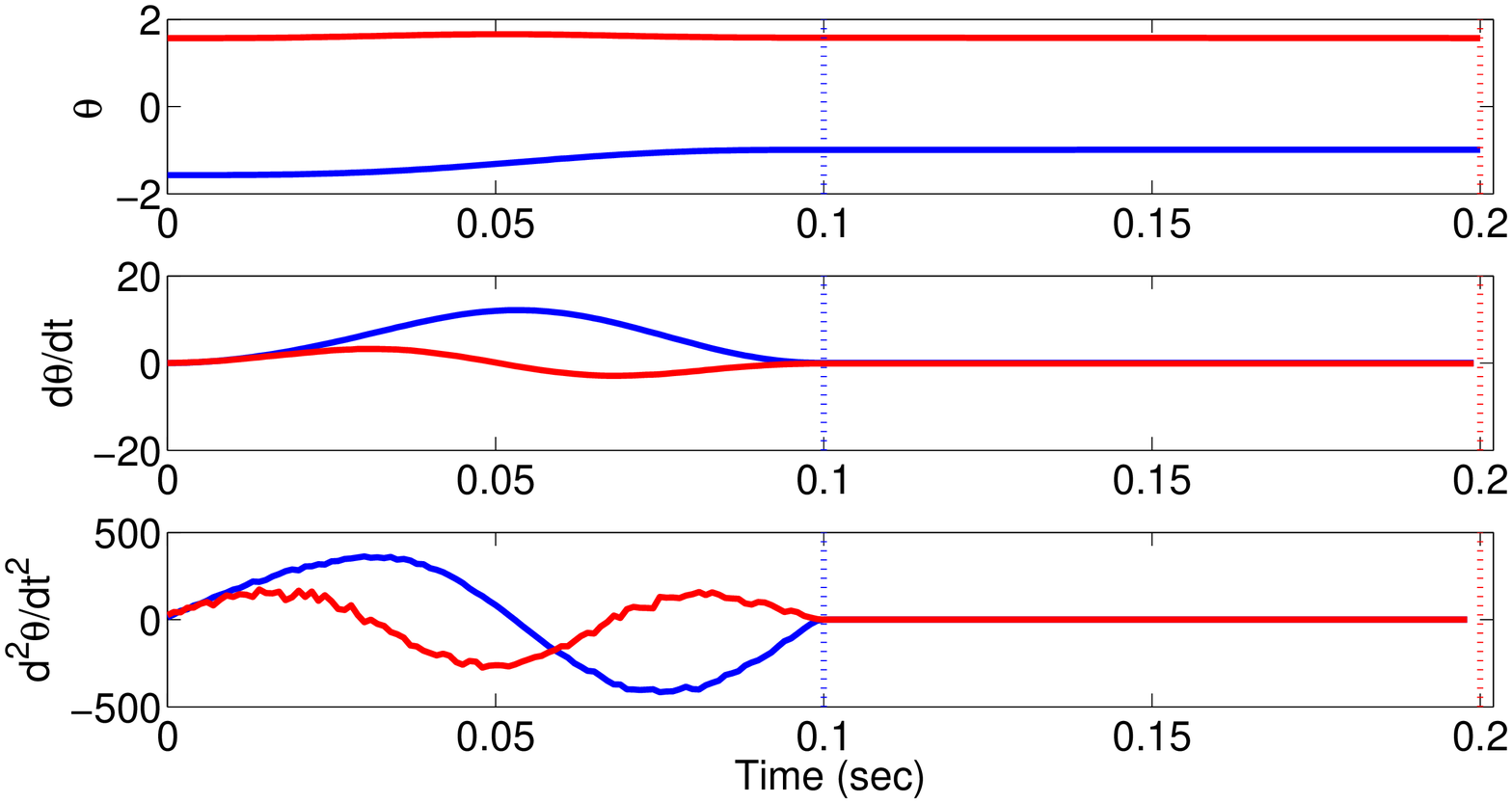}
\caption{Optimal control of straight-trajectory arm movement model \emph{%
with noise} with $T=100$ and $R=100$ (msec): The dynamics of the angles (a),
the angle velocities (b) and accelerations (c) (the blue solid curves for
those of $\protect\theta_{1}$ and the red solid curves for $\protect\theta%
_{2}$). The blue and red dash vertical lines stand for the start and end
time points of the duration of reaching the target.}
\label{factv1}
\end{figure}

\begin{figure}[ptb]
\centering{\includegraphics[width=0.9\textwidth]{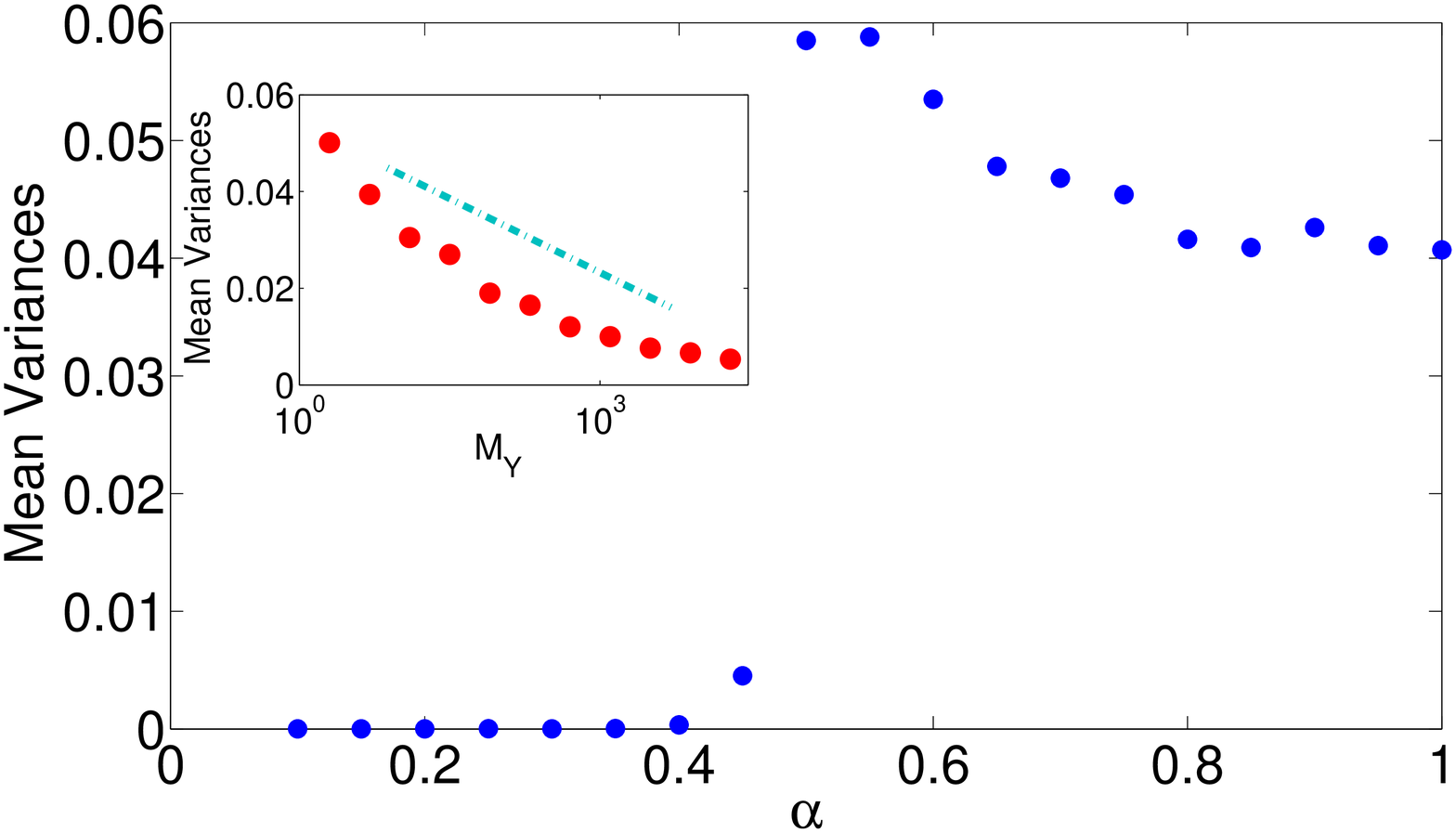}}
\caption{Performance of optimal control of straight-trajectory arm movement
model: Relationship between the executive error, measured by mean standard
variance, and dispersion index $\protect\alpha$ with $M_{Y}=20000$ and
Log-log plot (the inner plot) of the the relationship between executive
error and bound of the Young measure $M_{Y}$ with $\protect\alpha=0.25$
where the dash line is reference line with slope $-1/2+\protect\alpha=-0.25$%
, as shown in (\protect\ref{error_est}).}
\label{figs2}
\end{figure}

\begin{figure}[ptb]
\centering{\includegraphics[width=0.9\textwidth]{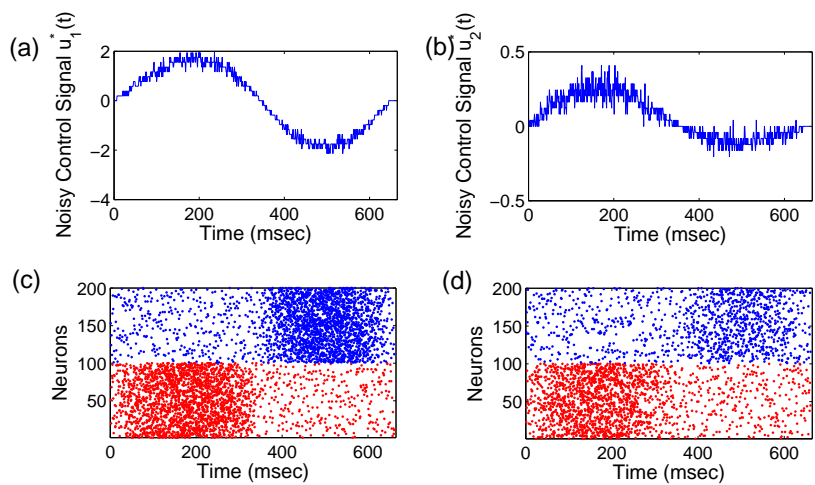}}
\caption{Spiking control of straight-trajectory arm movement model: (a).
Approximation of first component ($u_{1}^{*}(t)$) of the optimal Young
measure control signal by spike trains; (b). Approximation of second
component ($u_{2}^{\ast}(t)$) of the optimal Young measure control signal by
spike trains; (c) and (d). Approximation of optimal control signal $Q_{1,2}$
by spike trains of a set of neurons.}
\label{figs3}
\end{figure}

\end{document}